\documentclass[a4paper,10pt]{article}

\usepackage[margin=1.5in]{geometry}
\usepackage{amsmath}
\usepackage{amssymb}
\usepackage{mathrsfs}
\usepackage{enumerate}
\usepackage{relsize}

\newtheorem{theorem}{Theorem}[section]
\newtheorem{lemma}{Lemma}[section]

\newtheorem{definition}{{Definition}}[section]
\newtheorem{proposition}{Proposition}[section]
\newtheorem{remark}{Remark}[section]

\newcommand{\qed}{\hfill \mbox{\raggedright \rule{.07in}{.1in}}}
\newenvironment{proof}{\vspace{1ex}\noindent{\bf Proof.}\hspace{0.5em}}
	{\hfill\qed\vspace{1ex}}

\newcommand{\mathref}[1]{\ifmmode\mathrm{(\ref{#1})}\else(\ref{#1})\fi}
\newcommand{\mref}[1]{\ifmmode\mathrm{(\ref{#1})}\else(\ref{#1})\fi}

\newcommand{\domain}{\Omega}
\newcommand{\bound}{\Gamma}
\newcommand{\normal}{n}
\newcommand{\RR}{\mathbb R}
\newcommand{\intd}{\int_{\domain}}
\newcommand{\intb}{\int_{\bound}}
\newcommand{\nablab}{{\nabla_{\bound}}}
\newcommand{\trib}{\Delta_{\bound}}
\newcommand{\vX}{\vec{X}}
\newcommand{\vV}{{\vec V}}
\newcommand{\vW}{{\vec W}}
\newcommand{\divb}{{\rm{div}_{\bound}}}
\newcommand{\ddiv}{{\rm div}}

\begin{document}

\title{Shape Calculus for Shape Energies \\ in Image Processing}

\author{G\"unay Do\u gan\textsuperscript{1}}

\thanks{\textsuperscript{1}
Theiss Research, La Jolla, CA 92037, USA, and
Applied and Computational Mathematics Division,
Information Technology Laboratory,
National Institute of Standards and Technology,
Gaithersburg, MD 20899, USA ({\tt gunay.dogan@nist.gov}).
Partially supported by NSF Grant DMS-0505454 and NIST Grant 70NANB10H046.}

\date{}

\maketitle

\begin{abstract}
Many image processing problems are naturally expressed
as energy minimization or shape optimization problems,
in which the free variable
is a shape, such as a curve in 2d or a surface in 3d.
Examples are image segmentation, multiview stereo reconstruction,
geometric interpolation from data point clouds. To obtain the
solution of such a problem, one usually resorts to an iterative
approach, a gradient descent algorithm, which updates
a candidate shape gradually deforming it into the optimal shape.
Computing the gradient descent updates requires the knowledge of the first
variation of the shape energy, or rather the first shape derivative.
In addition to the first shape derivative, one can also utilize the second
shape derivative and develop a Newton-type method with faster convergence. Unfortunately, the knowledge of shape derivatives for shape energies 
in image processing is patchy. The second shape derivatives are known 
for only two of the energies in the image processing literature 
and many results for the first shape derivative are limiting,
in the sense that they are either for curves on planes, or developed
for a specific representation of the shape or for a very specific
functional form in the shape energy. In this work, these limitations 
are overcome and the first and second shape derivatives are computed for
large classes of shape energies that are representative
of the energies found in image processing. Many of the formulas we
obtain are new and some generalize previous existing results. These 
results are valid for general surfaces in any number of dimensions. 
This work is intended to serve as a cookbook for researchers who 
deal with shape energies for various applications in image processing 
and need to develop algorithms to compute the shapes minimizing 
these energies.
\end{abstract}

\section{Introduction}\label{S:intro}

Many image processing tasks are expressed as energy
minimization problems in which the free variable is a shape,
such as a curve in 2d or a surface in 3d, because the shape
is a geometric representation for the object or the region
of interest in the data. Examples of such tasks are image segmentation \cite{Aubert-etal-03,Caselles-Kimmel-Sapiro-97,Chan-Vese-01,Chan-Vese-PWConst},
surface regularization \cite{Clarenz-etal-04,Schoenemann-Kahl-Cremers-09},
geometric interpolation of data point clouds \cite{Zhao-etal-00}
and multiview stereo reconstruction
\cite{Faugeras-Keriven-98,Jin-Yezzi-Soatto-03,Kolev-Pock-Cremers-10}.
In these problems, one defines an appropriate
shape energy $J(\bound)$ that depends on the shape $\bound$, and the
shape energy is designed such that its minimum corresponds to a solution
of the image processing problem at hand. For example, as an image
segmentation formulation (to locate distinct objects, regions or
their boundaries in images),
one can choose to use the Geodesic Active Contour model
\cite{Caselles-Kimmel-Sapiro-97,Caselles-Kimmel-Sapiro-97-3d}, which is a weighted integral over candidate
curves or surfaces $\bound$, as the shape energy,
\begin{equation}\label{E:gac-energy00}
J(\bound) = \intb g(x) dS,
\end{equation}
where $g(x)$ is an image-based weight function,
or one can choose the following variant of Mumford-Shah functional \cite{Chan-Vese-PWConst}
\begin{equation}\label{E:Chan-Vese-energy00}
J(\bound) = \frac{1}{2} \sum_{i=1,2} \int_{\domain_i} (I(x)-c_i)^2 dx
+ \nu \intb dS, \qquad
c_i = \frac{1}{|\domain_i|}\int_{\domain_i} I(x) dx,
\end{equation}
where $I(x)$ is the image function and $\domain_1,\domain_2$ are the
inside and outside regions of the curve $\bound$.
Then the curve minimizing the energy is a valid solution of 
segmentation problem.

The energy minimization formulation comes naturally for many
image processing problems, because it gives a straight-forward
way to penalize unwanted configurations of the shapes considered
and to encourage the good configurations, especially when
the problem has a data-fitting component or
when the problem is hard to formulate in a direct
manner. For example, the energy~\eqref{E:Chan-Vese-energy00}
consists of a data term, which is minimized for curves that
separate the image into regions of constant intensity.
It also has a geometric penalty term, a length integral,
which favors shorter curves over longer curves and acts as
a regularizer in noisy images.

In order to find the minimizer $\bound^*$ of a given shape energy, one
needs to implement an energy minimization or shape optimization algorithm.
The shape optimization algorithms usually work iteratively; they start
from an initial shape $\bound_0$ and deform the shape through several
iterations with a velocity field $\vV$ until a minimum of the energy
is achieved. Thus, a crucial step to solve the minimization problem
is the computation of the gradient
descent velocities $\vV$ at each iteration, namely deformation velocities
that decrease the energy of the shape $\bound$. This requires understanding
how a deformation of $\bound$ induced by a given velocity $\vV$
changes the energy $J(\bound)$. An analytical tool that gives us this
information is the shape derivative $dJ(\bound;\vV)$ 
\cite{Delfour-Zolesio-01,Henrot-Pierre-05,Simon-Murat-76,Sokolowski}.
It tells us the change in the energy $J(\bound)$ of the shape $\bound$ when
$\bound$ is deformed by $\vV$ (see Section~\ref{S:shape-calculus}
for a rigorous definition of the shape derivative).
If a given velocity $\vV$ satisfies $dJ(\bound;\vV)<0$,
then $\vV$ is a gradient descent velocity. If the shape derivative is
zero for all $\vV$, then the shape $\bound$ is a stationary point,
possibly a minimal shape.
The shape derivative concept enables us to compute gradient
descent velocities for a given shape $\bound$ and its energy
$J(\bound)$ in a straight-forward manner. 
We review this briefly in Section~\ref{S:grad-descent}.
The gradient descent velocity can be computed in other ways too,
but we advocate shape differentiation in this paper as it is a very
powerful technique and is widely applicable.

We can differentiate the energy more than once, and obtain
the second shape derivative $d^2J(\bound;\vV,\vW)$, which is
defined with respect to two perturbation velocities $\vV,\vW$
\cite{Delfour-Zolesio-01,Novruzi-Pierre-02,Sokolowski}.
The second shape derivative gives us second order shape 
sensitivity information. It can be used to perform stability 
analysis of a given stationary point \cite{Dambrine-Pierre-00} 
and it can tell us
whether or not the stationary point is a minimum. Moreover
it can be used to design fast Newton-type minimization schemes,
which converge in fewer iterations.
Only a few such schemes exist in image processing
\cite{Bar-Sapiro-09,Dogan-Morin-Nochetto-08,Hintermueller-Ring-03,Hintermueller-Ring-04}
(more examples can be found in other areas of science
and engineering 
\cite{Bui-Ghattas-12,Eppler-Harbrecht-06,Goto-Fujii-90,Novruzi-Roche-00}).
The existence of only a few schemes is due to the fact 
that the \emph{explicit} formulas for the second shape derivatives
of most energies in image processing are not known.

{\bf Contributions.}
In this paper, we use the shape differentiation methodology
to derive the shape derivatives of large classes of shape
energies. We aim our results to be as comprehensive as
possible, so that this work will serve as a cookbook for
the researchers who need to analyze shape energies and
design algorithms that solve image processing problems
using shape energies. In Section~\ref{S:shape-energies},
we list the classes of shape energies
that we consider and give examples of how they are used
in the literature. The first shape derivatives or other
derivatives of equivalent use have been derived
for some of these energies. In some cases, these previous results
are specific to a certain dimension, for example, curves in 2d
or surfaces in 3d. In some cases, the results are specific
to a certain geometric representation, such as a parametric 
surface. In most cases, the \emph{explicit} formulas 
for the second shape derivative are not known.
In fact, to our knowledge, the second shape derivatives
have been computed for only two energies \cite{Hintermueller-Ring-03,Hintermueller-Ring-04}
in image processing (explicit formulas for some energies
relevant to other application areas can be found in literature \cite{Bui-Ghattas-12,Eppler-Harbrecht-06,Goto-Fujii-90,Novruzi-Roche-00})).
In this work, we derive the first shape derivatives
for all the classes of energies that we list
in Section~\ref{S:shape-energies} and
the second shape derivatives except for two of the energies
(see Tables~\ref{T:deriv-summary1}, \ref{T:deriv-summary2}
for a summary of the results, the new formulas derived 
that are previously unknown are denoted by $(\star)$).
These new formulas are the main contribution of this paper.
They are derived and laid out explicitly and are intended
to serve researchers in image processing.
These results are valid for hypersurfaces in any number
dimensions and do not depend on the representation
of the shape (parametric, level set or other).
The only limitation of these results is that in their
basic form, they are valid for closed surfaces or surfaces
whose boundaries are on the image domain boundary. Other types of
surfaces, such as open surfaces with boundaries inside the
image domain or surfaces with junctions,
require special consideration that we do not
include in this paper.

The emphasis in this paper is on deriving the shape derivatives
(assuming as much smoothness as needed of the shapes or the 
functions). Naturally, the existence of these shape derivative
rely on certain differentiability requirements, and these may
not be easy to satisfy in some practical situations. 
This will depend on the application and needs to be addressed
on a case-by-case basis; therefore, such questions are not
addressed in this paper. Moreover, existence of the shape 
derivatives does not imply the existence of the minimal shapes.
This is a critical question that one must ask before using 
the shape derivatives to compute the minimum of the shape 
energies. For more information on existence and uniqueness
of minimal shapes in shape optimization, we refer the reader
to \cite{Bucur-Buttazzo-Henrot-98,Bucur-Buttazzo-05}.

\textbf{Outline.}
We start with Section~\ref{S:shape-energies} explaining 
the major classes of shape energies used in image processing.
Our goal is to compute the shape derivatives for these energies.
In Section~\ref{S:shape-calculus}, we introduce some basic
differential geometry and some results and definitions
from shape differential calculus. In Section~\ref{S:shape-derivs-of-energies},
we use these results to compute the first and second shape
derivatives of the energies introduced in Section~\ref{S:shape-energies}.
These shape derivatives are the main contribution
of the paper and are summarized in
Tables \ref{T:deriv-summary1}, \ref{T:deriv-summary2}.
We conclude the paper with Section~\ref{S:grad-descent},
where we briefly review how the shape derivatives can be
used to compute gradient descent velocities for given shapes
and energies.

\section{Shape Energies in Image Processing}\label{S:shape-energies}

The search for geometric entities or geometric descriptions
for objects based on given images is a main theme in image 
processing. Thus, researchers in this field are constantly
devising new shape energies to address their problems, making
it a very fertile field for applications of shape optimization.
We find it useful to consider the numerous shape energies
in image processing in four main classes: minimal surface energies,
energies with integrals, higher-order energies, energies with
PDEs (partial differential equations).
Hybrids from these classes and exceptions are possible.
These shape energy classes are explained below with examples
from the literature. They are the starting point for the shape derivative
calculations in Section~\ref{S:shape-derivs-of-energies}.

\textbf{Minimal surface energies:}
The first example of a shape energy in image processing
was the Geodesic Active Contour Model proposed for
image segmentation by Caselles, Kimmel and
Sapiro in \cite{Caselles-Kimmel-Sapiro-97,Caselles-Kimmel-Sapiro-97-3d}
\footnote{Although the Snakes model of Kass, Witkin, Terzopoulos
\cite{Kass-etal-88} can also be viewed as a first example, it is not
truly a shape energy, because the value of the energy depends on
the parametrization of the curve and it can be different for
different parametrizations even though the shape is the same.}.
The main idea of their work was to try to fit 
a curve or a surface to the edges
of an object in the image. For this, they used an edge indicator
function $g(x)$ such that $g\approx 0$ on edges and $g\approx 1$
elsewhere and tried to compute a surface minimizing the following
energy
\begin{equation}\label{E:gac-energy0}
J(\bound) = \intb g(x) dS.
\end{equation}
One can sometimes add an area or volume integral
$\intd g(x) dx$ to~\eqref{E:gac-energy0} to speed up the computations
and to facilitate detection of concavities. Minimal surfaces
computed from the energy~\eqref{E:gac-energy0}
make very satisfactory segmentations
as they give continuous and smooth representations of the boundaries
of objects or regions in the given images. Thus, the model~\eqref{E:gac-energy0}
is very popular and is widely implemented. The implementation is
based on the first shape derivative or the first variation of
the energy. Only recently the second shape derivative
for~\eqref{E:gac-energy0} was computed by Hinterm\"uller and
Ring \cite{Hintermueller-Ring-03} and was used to devise a second order
minimization method resulting in faster convergence.

The energy~\eqref{E:gac-energy0} is isotropic, i.e.\! it does not
depend on the orientation or the normal of surface.
In \cite{Kimmel-Bruckstein-03}, Kimmel and Bruckstein
proposed an anisotropic energy that fits the general form
\begin{equation}\label{E:anisotropic-energy0}
J(\bound) = \intb g(x,\normal) dS.
\end{equation}
By setting $g=\langle \nabla I,\normal \rangle$, they aimed to better
align solution curves with object boundaries in images and were able
to attain improved segmentations. Before \cite{Kimmel-Bruckstein-03},
the energy~\eqref{E:anisotropic-energy0} had been used by Faugeras
and Keriven for multiview stereo reconstruction
\cite{Faugeras-Keriven-98} (later by Jin et al.\! in \cite{Jin-Yezzi-Soatto-03}
and by Kolev et al.\! \cite{Kolev-Pock-Cremers-10}).
The first shape derivative of~\eqref{E:anisotropic-energy0}
for general n-dimensional surfaces has been known in the literature
for geometric flows \cite{Bellettini-Paolini-97,Deckelnick-Dziuk-Elliott-05}.
We derive the second shape derivative in this paper.

The key feature of the energy~\eqref{E:anisotropic-energy0}
is the dependence on the geometry through the normal of the surface.
The dependence may be through other geometric properties
of the surface as well, such as the mean curvature $\kappa$
\begin{equation}\label{E:curvature-energy0}
J(\bound) = \intb g(x,\kappa) dS.
\end{equation}
The integral~\eqref{E:curvature-energy0} is usually used
as part of a more involved energy, to impose higher regularity 
of the surface. Examples are the Willmore functional with 
$g=\frac{1}{2}\kappa^2$ \cite{Clarenz-etal-04,WillmoreBook} 
or $g=\kappa^p$ in \cite{Schoenemann-Kahl-Cremers-09}. 
Sundaramoorthi et al.\! noted in \cite{Sundar-etal-09} that
$g=\frac{1}{2}w(x)\kappa^2$ with an image-based weight $w(x)$
yielded better regularizations for image segmentation.
Another more general functional form of $g(\kappa)$ was used in
\cite{Droske-Bertozzi-10} to implement a corner-preserving
regularization energy.
The first shape derivative for \eqref{E:curvature-energy0}
and for the energy with the more general geometric weight
$g=g(x,\normal,\kappa)$ was derived in \cite{Dogan-Nochetto-12}
by Do\u{g}an and Nochetto. \\

\textbf{Energies with integrals:}
If one views a surface as the boundary separating different
regions in the image from each other (in order to identify
distinct regions), a logical approach
to designing the shape energy is to incorporate terms
that compare the properties of the regions across the boundary
and try to find surfaces maximizing the difference between the regions.
The characteristics of each region can be quantified by
computing the statistics of the image features in the region
\cite{Cremers-Rousson-Deriche-07}. The statistics computations
can often be expressed as various integrals over the regions.
This results in shape energies with weight functions
that depend on integrals over the regions. An example
is the following energy~\eqref{E:ChanVese-energy}
proposed by Chan and Vese in \cite{Chan-Vese-PWConst}. 
It aims to find a partitioning
of the image domain into a foreground region $\domain_1$
and a background region $\domain_2$ (inside and outside
$\bound$ respectively), each with distinct averages
$c_1, c_2$ of the image intensity $I(x)$ respectively.
\begin{equation}\label{E:ChanVese-energy}
J(\bound) = \frac{1}{2} \sum_{i=1,2} \int_{\domain_i} (I(x)-c_i)^2 dx
+ \nu \intb dS, \qquad
c_i = \frac{1}{|\domain_i|}\int_{\domain_i} I(x) dx.
\end{equation}

More general approaches to incorporating statistics
into the shape optimization formulation are described in
\cite{Cremers-Rousson-Deriche-07,Paragios-Deriche-02a}.
To develop a more general statistical formulation, one can consider
a Bayesian interpretation of the estimation problem
and try to maximize the a posteriori probability
$p(\{\domain_1,\domain_2\}|I)$, namely the likelihood
of having a certain partitioning $\{\domain_1,\domain_2\}$
given the image $I$ (multiple phases or regions 
$\{\domain_i\}_{i=1}^m$ are possible, but not considered
in this paper to simplify the presentation).
We can write $p(\{\domain_1,\domain_2\}|I)$ as
\begin{equation}\label{E:bayes-rule}
p(\{\domain_1,\domain_2\}|I) \propto p(I|\{\domain_1,\domain_2\}) \
p(\{\domain_1,\domain_2\}),
\end{equation}
and separate the a priori shape information $p(\{\domain_1,\domain_2\})$
from image-based cues encoded in $p(I|\{\domain_1,\domain_2\})$.
A common example of the a priori shape term would be
$p(\{\domain_1,\domain_2\}) \varpropto e^{-\nu|\bound|}$.
Assuming no correlation between labelings of regions,
one can simplify the conditional probability
\[
p(I|\{\domain_1,\domain_2\})
= p(I|\domain_1) p(I|\domain_2) = p_1(I) p_2(I) .
\]
Maximizing the probability~\eqref{E:bayes-rule} is equivalent to
minimizing its negative logarithm. Thus we end up with the following
energy
\begin{equation*}
J(\bound) = -\int_{\domain_1} \log p_1(I(x)) dx
-\int_{\domain_2} \log p_2(I(x)) dx + \nu \intb dS.
\end{equation*}
If the distributions are modeled as parametric ones,
with parameters $\theta_i$ for $p_i$, then the energy
can be rewritten as
\begin{equation}\label{E:stat-shape-energy}
J(\bound) = -\int_{\domain_1} \log p(I(x)|\theta_1) dx
-\int_{\domain_2} \log p(I(x)|\theta_2) dx + \nu \intb dS.
\end{equation}
The parameters $\theta_i$ depend on the form of the probability
density function and often involve integrals over the regions
$\domain_i$. For example, the Gaussian probability density function
has the form
$p_i(s) = \frac{1}{\sqrt{2\pi\sigma_i^2}} \exp\left(-\frac{(s-c_i)^2}{2\sigma_i^2}\right)$,
where the parameters $c_i, \sigma_i$ are computed by the integrals
$
c_i = \frac{1}{|\domain_i|} \int_{\domain_i} I(x) dx, \quad
\sigma_i^2 = \frac{1}{|\domain_i|} \int_{\domain_i} (I(x)-c_i)^2 dx.
$ \\
It is not hard to see that we can concoct more complicated
statistical formulations where shape energies with integrals
play a central role.
Thus, shape energies with integral parameters have significant
use in image processing. The prototype for energies
with integrals is
\begin{equation}\label{E:domain-energy-w-integrals0}
J(\domain) = \intd g(x,I_w(\domain)) dx, \qquad
I_w(\domain) = \intd w(x) dx,
\end{equation}
or one whose weight function $g$ may depend on
multiple integrals
\begin{equation}\label{E:domain-energy-w-multiple-integrals0}
J(\domain) = \intd g(x,I_{w_1}(\domain),\ldots,I_{w_m}(\domain)) dx,
\qquad I_{w_i}(\domain) = \intd w_i(x) dx.
\end{equation}
The first shape derivatives for the energies
\eqref{E:domain-energy-w-integrals0},
\eqref{E:domain-energy-w-multiple-integrals0} were computed
by Aubert et al.\! in \cite{Aubert-etal-03}.
The second shape derivatives are computed in
Section~\ref{S:shape-derivs-of-energies} in this paper,
where we also deal with the case of nested integrals.
Similar to domain energies with integrals,
one can conceive of problems where it is
necessary to deal with surface energies with integral
parameters:
\begin{equation}\label{E:bound-energy-w-integral0}
J(\bound) = \intb g(x,I_w(\bound)) dS, \qquad
I_w(\bound) = \intb w(x) dS.
\end{equation}
The first and second shape derivatives for the energy \eqref{E:bound-energy-w-integral0}
are not available in the literature and
are computed in Section \ref{S:shape-derivs-of-energies}.\\

\textbf{Higher order energies:}
Some image processing problems require encoding nonlocal
interactions between points of a surface or a domain.
Shape energies involving such interactions may be written
as higher order integrals over the surface or the domain.
For example, an energy that encodes the interactions between
any two points of a surface or a domain would have the form
\begin{equation}\label{E:h-o-energies0}
J(\bound) = \intb\intb g(x,y) dS(x) dS(y), \qquad
J(\domain) = \intd\intd g(x,y) dx dy.
\end{equation}
The weight function $g(x,y)$ describes the
nature of the interaction between the points $x$ and $y$.
If we want to account for nonlocal interactions of more points,
say three, this can be formulated as a multiple integral
with even higher order, \\
like $\intd\intd\intd g(x,y,z)\ dx dy dz$.

Examples of higher order shape energies are not many in
image processing, but they have been used successfully in applications
such as road network extraction from images \cite{Rochery-Jermyn-Zerubia-06}
and topology control of curves in image segmentation
\cite{LeGuyader-Vese-08,Sundar-Yezzi-07}. In \cite{Sundar-Yezzi-07},
Sundaramoorthi and Yezzi used the following shape energy
to prevent curves from changing topology by merging or
splitting:
\begin{equation}\label{E:Sundar-Yezzi-energy0}
J(\bound) = \intb\intb \frac{1}{|x-y|^\gamma} dS(x) dS(y), \quad \gamma > 0.
\end{equation}
They added the energy~\eqref{E:Sundar-Yezzi-energy0} as an additional
term to their segmentation energy. Note that the value of the integral
~\eqref{E:Sundar-Yezzi-energy0} blows up as different parts of curve
get close to each other, hence a topological change is prevented.
In \cite{LeGuyader-Vese-08}, Le Guyader and Vese accomplished
the same goal by using a double integral over the domain $\domain$,
instead of the surface $\bound$. Rochery et al.\! proposed
higher order active contours in \cite{Rochery-Jermyn-Zerubia-06},
as a general formulation with multiple integrals over curves
(but not general surfaces).
They derived the first variation of the shape energy and illustrate
its use with an application in road network detection.
In this paper, the shape energies~\eqref{E:h-o-energies0}
are considered for general surfaces $\bound$ and domains
$\domain$ in $\RR^d$ and their first and second shape derivatives
are derived for general weight functions $g(x,y)$ (note that
the second shape derivatives are not known and first variations
are reported for specific $g(x,y)$ in previous work).
We also explain how shape derivatives for energies
with order higher than two can be derived. \\

\textbf{Energies with PDEs:}
Large classes of images can be modeled as piecewise smooth
functions with some discontinuities.
For such images, the problems of image segmentation
and image regularization can be formulated as finding
the discontinuity set $K$ and approximating the image
intensity function $I$ with a smooth function $u$ on
the remaining parts $D-K$ of the image domain $D\subset \RR^d$.
Mumford and Shah proposed minimizing the following
energy for this purpose \cite{Mumford-Shah-89}
\begin{equation}\label{E:original-MS-energy}
J(K,u) = \frac{1}{2}\int_D (I-u)^2
+ \frac{\mu}{2}\int_{D-K} |\nabla u|^2 + \nu |K|,
\quad \mu,\nu > 0
\end{equation}
The set of discontinuities that is included in the
formulation~\eqref{E:original-MS-energy} is very general
and can include cracks and triple junctions. Therefore,
a direct numerical realization of the minimization
of~\eqref{E:original-MS-energy} is not
practical. For this reason, Chan and Vese proposed
an alternative energy in \cite{Chan-Vese-01}
\begin{equation}\label{E:MS-energy0}
J(\bound) = \frac{1}{2} \sum_{i=1}^2 \int_{\domain_i} \left( (u_i-I)^2
+ \mu |\nabla u_i|^2 \right) dx + \nu \int_\bound dS,
\end{equation}
where $u_1, u_2$ are the smooth approximations to the image
$I$ computed by
\begin{equation}\label{E:MS-pde0}
-\mu \Delta u_i + u_i = I \ \textrm{in} \ \domain_i, \qquad
\frac{\partial u_i}{\partial\normal_i} = 0 \ \textrm{on} \ \partial \domain_i,
\end{equation}
and $\domain_1, \domain_2$ are the domains inside and outside
the surface $\bound$ respectively. In \cite{Chan-Vese-01},
the surface $\bound$ was
represented with a level set function. Thus cracks and triple
junctions were excluded (a method to represent cases with junctions
was proposed in \cite{Vese-Chan-02} using multiple level set functions).
Chan and Vese implemented a gradient descent method  based on
the first variation of \eqref{E:MS-energy0}. In \cite{Hintermueller-Ring-04},
Hinterm\"uller and Ring derived the first and second shape
derivatives of \eqref{E:MS-energy0}. The second shape derivative
was used in \cite{Dogan-Morin-Nochetto-08},\cite{Hintermueller-Ring-04}
to develop fast Newton-type minimization methods
for \eqref{E:MS-energy0}. A domain energy of the form
$\intd g(x,u,\nabla u) dx$ with a Neumann PDE like 
\eqref{E:MS-pde0} was considered Goto and Fujii in \cite{Goto-Fujii-90},
where they derived the first and second shape derivatives.

In \cite{Brox-Cremers-09}, Brox and Cremers gave a statistical 
interpretation of the Mumford-Shah functional. They started
from a Bayesian model for segmentation and introduced a local
Gaussian probability density function for image intensity
in each region $\domain_i$:
$p_i(I(x),x) = \frac{1}{\sqrt{2\pi}\sigma_i(x)} 
\exp\left( -\frac{(I(x)-c_i(x))^2}{2\sigma_i(x)^2} \right)$
with spatially varying mean $c_i(x)$ and variance $\sigma_i(x)$.
Then they obtained the following extended Mumford-Shah
functional
\begin{equation}\label{E:Cremers-MS-energy}
\begin{aligned}
J_{BC}(\bound,\{c_i\},\{\sigma_i\})
= &\sum_i \int_{\domain_i} \left( \frac{(I(x)-c_i(x))^2}{2\sigma_i(x)^2}
+ \frac{1}{2}\log(2\pi\sigma_i(x)^2) \right) dx \\
&+ \frac{\mu}{2}\sum_i\int_{\domain_i} \left( |\nabla c_i(x)|^2
+ |\nabla\sigma_i(x)|^2 \right) dx + \nu\intb dS,
\end{aligned}
\end{equation}
by linking their model with the filtering theory of Nielsen et al.\!
\cite{Nielsen-etal-97}. We generalize Brox and Cremer's model
and write 
\begin{equation}\label{E:gen-MS-energy0}
J_0(\bound,\{u_{ki}\}) = \sum_i \int_{\domain_i} \left( f(x,\{u_{ki}\})
+ \frac{\mu}{2} \sum_k |\nabla u_{ki}|^2 \right)dx
+ \nu \intb dS,
\end{equation}
where $u_{ki}$ is the smooth of approximation of a $k^{th}$
data channel or statistical descriptor over region $\domain_i$,
and $f(x,\{u_{ki}\})$ denotes a coupled data term, for example,
$f(x,\{u_i,v_i\}) = \frac{(I(x)-u_i)^2}{2 v_i^2}$
in the case of \eqref{E:Cremers-MS-energy} or
$f(x,\{u_{1i},u_{2i},u_{3i}\}) = \Sigma_{k=1}^{3}(I_k(x)-u_{ki}(x))^2$
for color image segmentation. We write the optimality condition
of \eqref{E:gen-MS-energy0} with respect to $\{u_{ki}\}$
\begin{equation}\label{E:gen-MS-pde0}
-\Delta u_{ki} + f_{u_{ki}}(x,\{u_{li}\}) = 0 \ \mathrm{in} \ \domain_i,
\qquad \frac{\partial u_{ki}}{\partial \normal} = 0 
\ \mathrm{on} \ \partial\domain_i,
\end{equation}
and use the solution of \eqref{E:gen-MS-pde0}
to write the reduced shape energy 
$J(\bound) = J_0(\bound,\{u_{ki}(\bound)\}) $
\begin{equation}\label{E:gen-MS-energy1}
J(\bound) = \sum_i \int_{\domain_i} \left( f(x,\{u_{ki}(\bound)\})
+ \frac{\mu}{2} \sum_k |\nabla u_{ki}(\bound)|^2 \right) dx
+ \nu \intb dS.
\end{equation}
The energy \eqref{E:gen-MS-energy1} is more general than 
\eqref{E:MS-energy0} and \eqref{E:Cremers-MS-energy}, 
and its first and second shape derivatives are not known;
they are derived in Section~\ref{S:energies-with-pdes} 
in this paper.

In \eqref{E:MS-energy0}, the role of the elliptic PDE was
in computing a piecewise smooth approximation $u$ to the image data
$I$ on the domains $\domain_i$. One can as well be
interested in finding a smooth approximation
to data defined on the surface $\bound$. This requires
using an elliptic PDE defined on the surface $\bound$.
Such a formulation was proposed in \cite{Jin-Yezzi-Soatto-03}
by Jin, Yezzi and Soatto for stereoscopic reconstruction of
3d objects and their surface reflectance from 2d projections
of the objects. The shape energy they used is essentially
the following
\begin{equation}\label{E:Jin-energy0}
J(\bound) = \frac{1}{2}\intb (u(x)-d(x))^2 dS
+ \frac{\mu}{2} \intb |\nablab u|^2 dS + \nu\intb dS,
\end{equation}
where $u$ is computed from the surface PDE, $-\mu\trib u + u = d$
on $\bound$ (see \S\ref{S:diff-geom} for the definition of
the surface gradient $\nablab$ and the surface Laplacian $\trib$),
and $d(x)$ is some data function based on the
2d images of the 3d scene. Jin, Yezzi and Soatto considered
the parametric representation of a 2d surface in 3d in order
to derive the first variation of the surface energy.
Then they implemented a gradient descent algorithm using the
level set method. The first shape derivative of~\eqref{E:Jin-energy0}
for general surfaces in any number of dimensions,
to our knowledge, is not available in literature. 
We consider the following more general energy
\begin{equation}\label{E:gen-Jin-energy0}
J(\bound) = \intb f(x,\{u_k(\bound)\}) dS
+ \frac{\mu}{2} \sum_k \intb |\nablab u_k(\bound)|^2 dS + \nu \intb dS,
\end{equation}
where $\{u_k\}_{k=1}^m$ are computed from the optimality PDE:
$-\mu \trib u_k + f_{u_k}(x,\{u_l\}) = 0$.
We compute the first shape derivative of \eqref{E:gen-Jin-energy0}
in Section~\ref{S:energies-with-pdes}.
This result includes the case of \eqref{E:Jin-energy0} as well.
Unlike \cite{Jin-Yezzi-Soatto-03}, it is not restricted 
to parameterized surfaces and is valid in any
number of dimensions.


\begin{table}[h!p]
\begin{tabular}{|l|}
\hline
\multicolumn{1}{|c|}{\bf Summary of results (1)} \\
\hline
\hline
\smallskip
{\bf Minimal surface energies} \\
\hline
\smallskip
Geodesic active contour energy:
$\displaystyle J(\bound) = \intb g(x) dS + \gamma\intd g(x) dx,$
\\
\cite{Caselles-Kimmel-Sapiro-97,Caselles-Kimmel-Sapiro-97-3d,Hintermueller-Ring-03}:
$\displaystyle
dJ(\domain;V)= \intb \left((\kappa+\gamma)g(x)
  + \partial_\normal g(x)\right) V dS,
\qquad \quad (\partial_\normal g: \mathrm{normal \ derivative})
$
\\
\begin{tabular}{rl}
\cite{Hintermueller-Ring-03}: $\quad d^2J(\bound;V,W) $ 
& $\displaystyle = \intb  g \nablab V \cdot \nablab W dS$  
\\
{} &
$\displaystyle \ + \intb \left( \partial_{\normal\normal} g
+ (2\kappa + \gamma) \partial_\normal g
+ (\kappa^2 - \sum \kappa_i^2 + 2\gamma\kappa) g \right) V W dS.
$
\end{tabular}
\smallskip
\\
\hline
\smallskip
Normal-dependent surface energy:
$\displaystyle J(\bound) = \intb g(x,\normal) dS,$
\\
\cite{Bellettini-Paolini-97,Deckelnick-Dziuk-Elliott-05}:
$\displaystyle
dJ(\bound;V)= \intb \left(\kappa g + \partial_\normal g
+ \divb(g_y)_\bound \right) V dS,
\qquad \quad (g_y: \mathrm{gradient \ w.r.t. \ 2nd \ arg.})
$
\\
$\displaystyle
\begin{aligned}
(\star):\
d^2J(\bound;V,W) =& \intb \nablab V \cdot \left( (g-g_y \cdot \normal)Id +
g_{yy}\right) \cdot \nablab W  dS \\
&+ \intb \left(\partial_{\normal\normal}g+2\kappa \partial_\normal g
+ (\kappa^2 - \sum \kappa_i^2) g \right)VW dS \\
&- \intb (\kappa g_y + \normal^T g_{xy})
\cdot ( \nablab W \, V + \nablab V \, W ) dS .
\end{aligned}
$
\smallskip
\\
\hline
\smallskip
Curvature-dependent surface energy:
$\displaystyle J(\bound) = \intb g(x,\kappa) dS,$ \\
\cite{Dogan-Nochetto-12}:
$\displaystyle
dJ(\bound;V) = \intb \left( -\trib(g_z)
+ g \kappa - g_z \sum \kappa_i^2 + \partial_\normal g \right) V dS,
\qquad \ (g_z: \mathrm{deriv. \ w.r.t. \ 2nd \ arg.})
$
\\
$
\qquad\qquad\quad
\trib(g_z) = \trib g_z + g_{zz} \trib\kappa + g_{zzz} |\nablab\kappa|^2
+ (g_{zxz} + g_{zzx})\cdot\nablab\kappa.
$
\smallskip
\\
\hline
\hline
\smallskip
{\bf Energies with integral parameters} (more results in \S\ref{S:energies-with-integrals}) \\
\hline
\smallskip
Energies with domain integrals:
$\displaystyle J(\domain) = \intd g(x,I_w(\domain))dx, \qquad I_w(\domain)=\intd w(x)dx,$
\\
\cite{Aubert-etal-03}:
$\displaystyle
dJ(\domain;V) = \intb \left( g(x,I_w(\domain))
+ I_{g_p}(\domain) w(x) \right) V dS,
\qquad \ (g_p: \mathrm{deriv. \ w.r.t. \ 2nd \ arg.})
$
\\
$\displaystyle
\begin{aligned}
(\star):\
d^2J(\domain;V,W) =& \intb \left( \partial_\normal g + I_{g_p} \partial_\normal w
+ \kappa(g + I_{g_p} w) \right) V W dS \\
&+ \intb g_p V dS \intb w W dS + \intb w V dS \intb g_p W dS
+ I_{g_{pp}} \intb w V dS \intb w W dS.
\end{aligned}
$
\smallskip
\\
\hline
\smallskip
Energies with surface integrals:
$\displaystyle J(\bound) \intb g(x,I_w(\bound))dS, \qquad I_w(\bound)=\intb w(x)dS,$
\\
$\displaystyle
(\star):\
dJ(\bound;V) = \intb \left( (g + I_{g_p} w) \kappa
+ \partial_\normal g + I_{g_p} \partial_\normal w \right) V dS,
\qquad \ (g_p: \mathrm{deriv. \ w.r.t. \ 2nd \ arg.})
$
\\
$\displaystyle
\begin{aligned}
(\star):\
d^2J(\bound;V,W) =& \intb (g + I_{g_p} w) \nablab V \cdot \nablab W dS \\
&+ \intb \left( \partial_{\normal\normal} g +  I_{g_p} \partial_{\normal\normal} w +
2(\partial_\normal g + I_{g_p}\partial_\normal w)\kappa \right) V W dS \\
&+ \intb \left( \partial_\normal g_p + g_p \kappa \right) V dS \intb \left(
\partial_\normal w + w \kappa \right) W dS \\
&+ \intb \left( \partial_\normal w + w \kappa \right) V dS \intb \left(
\partial_\normal g_p + g_p \kappa \right) W dS \\
&+ I_{g_{pp}} \intb \left( \partial_\normal w + w \kappa \right) V dS
\intb \left( \partial_\normal w + w \kappa \right) W dS.
\end{aligned}
$
\smallskip
\\
\hline
\end{tabular}
\caption{See the caption of Table~\ref{T:deriv-summary2}
for an explanation of the labels $(\star)$,$[\#]$.}
\label{T:deriv-summary1}
\end{table}

\begin{table}[h!p]
\begin{tabular}{|l|}
\hline
\multicolumn{1}{|c|}{\bf Summary of results (2)} \\
\hline
\hline
\smallskip
{\bf Higher order energies} \\
\hline
\smallskip
Higher order domain energies:
$\displaystyle J(\domain) = \intd\intd g(x,y) dy dx,$
\\
($\star$,\cite{LeGuyader-Vese-08}):
$\displaystyle
dJ(\domain;V) = \intb \intd\tilde{g}(x,y)dy V dS,
\qquad \left( \tilde{g}(x,y) = g(x,y)+g(y,x) \right)
$
\\
$\displaystyle
\begin{aligned}
(\star):\
d^2J(\domain;V,W) = &\intb\intb \tilde{g}(x,y) W(y) dS(y) V(x) dS(x) \\
&+ \intb \left( \kappa(x) \intd \tilde{g}(x,y)dy
+ \normal(x)\cdot \intd \tilde{g}_y(x,y)dy \right) V W dS.
\end{aligned}
$
\smallskip
\\
\hline
\smallskip
Higher order surface energies:
$\displaystyle J(\bound) = \intb\intb g(x,y) dS(y) dS(x),$
\\
($\star$,\cite{Rochery-Jermyn-Zerubia-06}):
$\displaystyle
dJ(\bound;V)= \intb \left(\kappa(x) \intb \tilde{g}(x,y) dS(y)
+\normal(x)\cdot\intb \tilde{g}_x(x,y) dS(y) \right) V(x) dS(x),
$
\\
$\displaystyle
\begin{aligned}
(\star):\
d^2J(\bound; &V,W) = \intb G(x,\bound)\nablab V\cdot \nablab W dS(x)
\qquad \qquad \left(G(x,\bound) = \intb \tilde{g}(x,y) dS(y) \right)
\\
&+ \intb \left( \normal^T G_{xx}(x,\bound)\,\normal
+ 2\kappa\, G_{x}(x,\bound)\cdot\normal
+ (\kappa^2 - \Sigma\kappa_i^2) G(x,\bound) \right) V W dS(x)
\\
&+ \intb \kappa \intb \tilde{g}\kappa\, W dS(y) V dS(x)
+ \intb \normal^T \intb \tilde{g}_{xy}\normal\, W dS(y) V dS(x)
\\
& + \intb \kappa \intb \tilde{g}_y \cdot\normal\, W dS(y) V dS(x)
+ \intb \normal\cdot\intb\tilde{g}_x\kappa\, W dS(y) V dS(x),
\end{aligned}
$
\smallskip
\\
\hline
\hline
\smallskip
{\bf Energies with PDEs} \\
\hline
$\displaystyle
\begin{array}{ll}
\mathrm{Energies \ with \ domain \ PDEs:} &
\displaystyle{
J(\domain) = \intd \left(f(x,\{u_k\}) 
+ \frac{\mu}{2} \sum_k |\nabla u_k|^2 \right) dx,}
\\
 &
-\mu\Delta u_l + f_{u_l} = 0 \ \mathrm{in} \ \domain, 
\ \frac{\partial u_l}{\partial\normal} = 0 \  \mathrm{on} \ \partial\domain, \ l=1,\ldots,m,
\end{array}
$
\\
($\star$,\cite{Brox-Cremers-09,Goto-Fujii-90,Hintermueller-Ring-04,Vese-Chan-02}):
$\displaystyle
dJ(\domain;V)= \intb \left(f(x,\{u_k\})
+ \frac{\mu}{2}\sum_k |\nablab u_k|^2 \right) V dS,
$
\\
\begin{tabular}{rl}
($\star$,\cite{Goto-Fujii-90,Hintermueller-Ring-04}): $\quad d^2J(\domain;V,W)$
& $\displaystyle
= \intb \left(f\kappa + \frac{\partial f}{\partial\normal} 
+ \mu \sum_k \nablab u_k^T \left( \frac{\kappa}{2}Id 
- \nablab\normal \right) \nablab u_k \right)  V W dS$ 
\\
{} & $\displaystyle \ 
+\intb \sum_k \left( f_{u_k} u_{k,W}' 
+ \mu \nablab u_k\cdot\nablab u_{k,W}' \right) V dS.
$
\end{tabular}
\\
\hline
$\displaystyle
\begin{array}{ll}
\mathrm{Energies \ with \ surface \ PDEs:} &
\displaystyle{
J(\bound) = \int_\bound \left( f(x,\{u_k\})
+ \frac{\mu}{2} \sum_k |\nablab u_k|^2 \right) dS,}
\\
 &
-\mu\trib u_l + f_{u_l} = 0 \ \mathrm{on} \ \bound, \ l=1,\ldots,m,
\end{array}
$
\\
($\star$,\cite{Jin-Yezzi-Soatto-03}):
$\displaystyle
dJ(\bound;V) = \intb \left(f \kappa
+ \frac{\partial f}{\partial\normal} 
+ \mu \sum_k \nablab u_k^T \left(\frac{\kappa}{2}Id 
- \nablab\normal\right)\nablab u_k \right) V dS.
$
\smallskip
\\
\hline
\end{tabular}
\caption{Next to each shape derivative formula is a label
specifying whether or not the formula is new and citation to its
source if it is not new. The label $(\star)$ indicates that
it is a new result, not available in the literature in any
form. The label $[\#]$ indicates that the result can be found
in reference $[\#]$. The label $(\star,[\#])$ indicates that
the result is new for general surfaces or domains in $\RR^d$
or for general choices of the weight functions
$g(x,\bound), g(x,\domain)$, but formulas for restricted
situations can be found in reference $[\#]$.
}
\label{T:deriv-summary2}
\end{table}

\section{Shape Differential Calculus}\label{S:shape-calculus}

In this section, we will reviews some basic differential
geometry that we will refer to throughout the paper. We will
prove some useful geometric formulas. Finally we will introduce
the shape derivative concept and describe some related results
that will enable us to differentiate the model energies
from Section~\ref{S:shape-energies}.

\subsection{Review of Differential Geometry}\label{S:diff-geom}

We assume that $\bound$ is a smooth orientable compact $(d-1)$ 
dimensional surface in $\RR^d$ without boundary.
Let us be given $h\in C^2(\bound)$ and a smooth extension $\tilde{h}$ of $h$,
$\tilde{h}\in C^2(U)$ and $\tilde{h}\vert_{\bound}=h$ on $\bound$
where $U$ is a tubular neighborhood of $\bound$ in $\RR^d$.
The {\it tangential gradient} $\nabla_\bound h$ of $h$ is
defined by:
$$
\nabla_\bound h=\big(\nabla \tilde{h}
- \partial_\normal\tilde{h}~\normal \big)\vert_{\bound},
$$
where $\normal$ denotes the unit normal vector to $\bound$
and $\partial_\normal \tilde{h} = \nabla\tilde{h}\cdot\normal$
is the normal derivative.
Similarly, given $\vW \in [C^1(\bound)]^d$ and its smooth extension
$\tilde{W}\in C^1(U)$, we define the \emph{tangential divergence}
of $\vW$ by
\begin{equation}
  \ddiv_{\bound}\vW= \big(\ddiv \tilde{W}-\normal\cdot D\tilde{W}\cdot \normal\big)
  \vert_{\bound},
\end{equation}
where $D\tilde{W}$ denotes the Jacobian matrix of $\tilde{W}$.
We also define the tangential gradient $\nablab \vW$ of $\vW$,
which is a matrix whose $i^{th}$ row is the tangential gradient
$\nablab \vW_i$ of the $i^{th}$ component of $\vW$.
Finally, the {\it tangential Laplacian} or {\it Laplace-Beltrami operator}
$\Delta_\bound$ on $\bound$ is defined as follows:
\begin{equation}\label{E:Laplace-Beltrami}
\Delta_\bound h = \ddiv_\bound(\nabla_\bound h)=
\big(\Delta\tilde{h} - \normal\cdot D^2\tilde{h}\cdot \normal
-\kappa ~\partial_\normal\tilde{h}\big)\vert_{\bound}.
\end{equation}

As seen above, in order to compute the full spatial
derivatives of surface functions and geometric quantities,
such as the normal $\normal$ and the mean curvature $\kappa$,
defined only on the surface $\bound$, we need to extend them to a tubular
neighborhood of $\bound$. This is accomplished using a signed
distance function representation of the surface $\bound$:
\begin{equation}\label{E:signed-dist}
b(x,\bound) = \left \{ \begin{array}{ll}
               \ \ \textrm{dist}(x,\bound) & \textrm{for }x\in \RR^d-\domain\\
               \ \ 0 & \textrm{for } x \in \bound \\
               -\textrm{dist}(x,\bound) & \textrm{for } x \in \domain
                \end{array}\right.
\end{equation}
where $\textrm{dist}(x,\bound) = \inf_{y\in\bound}|y-x|$
and $\domain$ is the domain enclosed by $\bound$.
Using \eqref{E:signed-dist} we extend the normal $\normal$,
the mean curvature $\kappa$ and the second fundamental form
$\nablab\normal$ as follows \cite[Chap. 8]{Delfour-Zolesio-01}:
\begin{equation}\label{E:geom-signed-dist}
\normal = \nabla b(x)|_\bound, \quad \kappa = \Delta b(x)|_\bound, \quad
\nablab\normal = D^2 b(x)|_\bound.
\end{equation}
The extensions~\eqref{E:geom-signed-dist} allow us to differentiate
$\normal, \kappa, \nablab\normal$ in the normal direction in addition
to the tangential direction, and the normal derivatives of the
normal $\normal$ and the mean curvature $\kappa$ are given by
\begin{equation}\label{E:normal-deriv-curv}
\partial_\normal \normal = 0, \qquad
\partial_\normal \kappa = -\sum_{i=1}^{d-1} \kappa_i^2,
\end{equation}
respectively, where $\kappa_i$ denote the principal curvatures 
of the surface.
It is easy to show, using the extension~\eqref{E:geom-signed-dist},
\[
\partial_\normal \normal = \nabla(\nabla b) \nabla b |_\bound
= D^2 b \nabla b |_\bound = 0,
\]
because $0 = \nabla(1) = \nabla(\nabla b \cdot \nabla b) = 2D^2 b \nabla b$.
The expression for $\partial_\normal \kappa$ can be computed similarly.
The proof can be found in \cite{Dogan-Nochetto-12}.

Note that Equation~\eqref{E:normal-deriv-curv}
holds only for parallel surfaces defined by the signed distance function.

\begin{lemma}\label{L:deriv-surf-f-u}
The following identities hold on $\bound$ for a function
$u$ of class $C^2$ defined in (or extended properly to)
a tubular neighborhood $U$ of the surface $\bound$,
\begin{align}
\frac{\partial}{\partial\normal}(\nablab u) &= D^2u\normal
- \normal^T D^2u\normal\normal \label{E:normal-deriv-of-tangential-grad}
\\
\nablab\left(\frac{\partial u}{\partial\normal}\right)
&= \normal^T D^2u + \nablab u^T\nablab\normal - \normal^T D^2u\normal\normal.
\label{E:tangential-grad-of-normal-deriv}
\end{align}
If the function $u$ is constant in the normal direction to $\bound$,
i.e.\! $\frac{\partial u}{\partial\normal}=0$ on $\bound$,
\begin{equation}\label{E:D2u.n=...}
\normal^T D^2u = -\nablab u^T\nablab\normal,
\end{equation}
\begin{equation}\label{E:mixed-normal-tangential-deriv}
\nablab\left(\frac{\partial u}{\partial \normal}\right) = 0, \qquad
\frac{\partial}{\partial \normal}\left(\nablab u \right)
= -\nablab\normal \nablab u.
\end{equation}
\end{lemma}
\begin{proof}
We start by computing the normal derivative of the tangential gradient.
For this, we resort to the signed distance function extension
~\eqref{E:geom-signed-dist} of the normal $\normal$.
\begin{align*}
\frac{\partial}{\partial\normal}(\nablab u)
&= \partial_{x_j}(u_{x_i} - u_{x_k}\normal_k\normal_i) \normal_j,
\qquad i=0,\ldots,d-1 \\
&= \left( \partial_{x_j}(u_{x_i}
- u_{x_k} b_{x_k}b_{x_i}) b_{x_j} \right)|_\bound, \qquad i=0,\ldots,d-1 \\
&= \left( u_{x_ix_j} b_{x_j} - u_{x_kx_j} b_{x_k}b_{x_i} b_{x_j}
- u_{x_k} b_{x_kx_j}b_{x_i} b_{x_j} - u_{x_k} b_{x_k}b_{x_ix_j} b_{x_j} \right)|_\bound \\
&= \left( u_{x_ix_j} b_{x_j} - u_{x_kx_j} b_{x_k}b_{x_i} b_{x_j} \right)|_\bound \\
& = D^2u\normal - \normal^TD^2u\normal\normal.
\end{align*}
Note that $b_{x_ix_j} b_{x_j} = b_{x_jx_i} b_{x_j}
= \frac{1}{2}\partial_{x_i} (b_{x_j}b_{x_j}) = 0$
because $|\nabla b|^2 = b_{x_j}b_{x_j} = 1$. \\
Now we compute the tangential gradient of the normal derivative,
\begin{align*}
\nablab \left( \frac{\partial u}{\partial\normal} \right)
&= \nablab \left( \nabla u\cdot\normal \right)
= \partial_{x_i}(u_{x_k}\normal_k)
- \partial_{x_j}(u_{x_k}\normal_k) \normal_j\normal_i, \qquad i=0,\ldots,d-1 \\
&= \left( \partial_{x_i}(u_{x_k}b_{x_k})
- \partial_{x_j}(u_{x_k}b_{x_k}) b_{x_j}b_{x_i} \right) \mathlarger{|}_\bound, \qquad i=0,\ldots,d-1 \\
&= ( u_{x_kx_i}b_{x_k} + u_{x_k}b_{x_kx_i}
- u_{x_kx_j}b_{x_k} b_{x_j}b_{x_i}
- u_{x_k} \underbrace{b_{x_kx_j} b_{x_j}}_{{}=0} b_{x_i} ) \mathlarger{|}_\bound\\
&= ( u_{x_kx_i}b_{x_k} + (u_{x_k}b_{x_kx_i}
- (u_{x_l} b_{x_l}) \underbrace{b_{x_k} b_{x_kx_i}}_{{}=0} )
- u_{x_kx_j}b_{x_k} b_{x_j}b_{x_i} )|_\bound \\
&= \normal^T D^2u + \nablab u^T\nablab\normal - \normal^T D^2u\normal\normal.
\end{align*}
We used the fact that $D^2b$ is symmetric and $\nablab\normal = D^2b|_\bound$. \\
To prove~\eqref{E:D2u.n=...}, we start with the assumption
$\frac{\partial u}{\partial\normal} = (\nabla u\cdot\nabla b)|_\bound=0$
and differentiate,
\begin{align*}
0 &= \partial_{x_j}(\nabla u\cdot\nabla b) = \partial_{x_j}(u_{x_i}b_{x_i}) 
     = u_{x_ix_j}b_{x_i} + u_{x_i}b_{x_ix_j} \\
  &= u_{x_ix_j}b_{x_i} + b_{x_ix_j}u_{x_i} - u_{x_k} b_{x_k}b_{x_i}b_{x_ix_j}
\qquad (b_{x_i}b_{x_ix_j} = 0) \\
  &= u_{x_ix_j}b_{x_i} + b_{x_ix_j}(u_{x_i} - u_{x_k} b_{x_k}b_{x_i}) \\
\Rightarrow \ 
&\nabla b^T D^2u = -(\nabla u - \nabla u\cdot\nabla b \nabla b)^T D^2b,
\end{align*}
which on the surface $\bound$ is equivalent to
\begin{equation}\label{E:D2u.n}
\normal^T D^2u = -\nablab u^T \nablab\normal.
\end{equation}
The identities~\eqref{E:mixed-normal-tangential-deriv}
follow trivially from~\eqref{E:D2u.n} substituted in \eqref{E:normal-deriv-of-tangential-grad},
\eqref{E:tangential-grad-of-normal-deriv}.
\end{proof}

\begin{proposition}[\textrm{\cite[Sect. 8.5.5, (5.27)]{Delfour-Zolesio-01}}]
\label{P:tangential-greens}
For a function $f \in C^1(\bound)$ and a vector $\vec{\omega} \in
C^1(\bound)^d$, we have the following tangential Green's formula
\begin{equation}\label{E:greens-formula}
\intb f \divb \vec{\omega} + \nablab f \cdot \vec{\omega} dS
= \intb \kappa f \vec{\omega} \cdot \nu dS.
\end{equation}
\end{proposition}
%

\subsection{Shape Differentiation}\label{S:shape-diff}

We would like to understand how a quantity depending on
a surface $\bound$ (or a domain $\domain$) changes when $\bound$
is deformed by a given velocity field. For this, we consider
a hold-all domain $\mathcal{D}$ (which may or may not be the image domain),
containing the surface $\bound$, and a smooth vector field
$\vV$ defined on $\mathcal{D}$. The vector field $\vV$ is used to
define the continuous sequence of perturbed surfaces $\{\bound_t\}_{t\geq 0}$,
with $\bound_0:=\bound$. Each point $X\in\bound_0$ follows
\begin{equation}\label{ode:mapping}
\frac{dx}{dt}=\vV(x(t)),\quad \forall t\in [0,T],
\qquad
x(0)=X.
\end{equation}
This defines the mapping $x(t,\cdot): X\in\bound \to x(t,X)\in \RR^d$
and the perturbed sets \\
$\bound_t=\{x(t,X):~~X\in\bound_0\}$
(similarly perturbations of domains
$\domain_t=\{x(t,X):~~X\in\domain_0\}$ for domains
$\domain(=\domain_0)$ contained by $\bound$).

Let $J(\bound)$ be a shape energy, namely a mapping that
associates to surfaces $\bound$ a real number.
The Eulerian derivative, or {\it shape derivative},
of the energy $J(\bound)$ at $\bound$ in the direction
of the vector field $\vV$, is defined as the limit
\begin{equation}
  dJ(\bound;\vV)=\lim_{t\to 0}
\frac{1}{t}\big(J(\bound_t)-J(\bound)\big).
\end{equation}
We define the shape derivatives $dJ(\domain;\vV)$ for domain
energies $J(\domain)$ similarly. For more information
on the concept of shape derivatives (including the definition
and other properties), we refer to the book
\cite{Delfour-Zolesio-01} by Delfour-Zolesio-01 and Zol{\'e}sio.

We now recall a series of results from shape differential
calculus in $\RR^d$.
\begin{lemma}[\textrm{\cite[Prop.2.45]{Sokolowski}}]\label{L:deriv-domain}
Let $\phi\in W^{1,1}(\RR^d)$ and $\domain\subset\RR^d$ be an open
and bounded domain with boundary $\bound=\partial\domain$ of class $C^1$.
Then the energy $ J(\domain)=\int_{\domain}\phi dx $
is shape differentiable. The shape derivative of $J(\domain)$ is given by
\begin{equation}
  dJ(\domain;\vV)=\intb\phi V dS,
\end{equation}
where $V=\vV\cdot\normal$ is the normal component of the velocity.
\end{lemma}

\begin{lemma}[\textrm{\cite[Prop. 2.50 and (2.145)]{Sokolowski}}]
\label{L:deriv-bound}
Let $\psi\in W^{2,1}(\RR^d)$ and $\bound$ be of class
$C^2$. Then the energy
$  \ J(\bound)=\intb\psi dS  \ $
is shape differentiable and the derivative
 \begin{equation}
   dJ(\bound;\vV)=\intb \big(\nabla\psi\cdot\vV +\psi \ddiv_{\bound}\vV\big)dS
   = \intb\big(\partial_\normal\psi + \psi \kappa\big) V dS,
 \end{equation}
depends on the normal component $V = \vV\cdot\normal$ of the velocity $\vV$.
\end{lemma}

Let us now consider more general energies $J(\bound)$.
Specifically we are interested in computing shape derivatives
for energies of the form
\begin{equation}\label{E:general-energies}
  J(\bound)=\intb \varphi(x,\bound)dS, \qquad
  J(\domain)=\intd \phi(x,\domain) dx,
\end{equation}
in which the weight functions $\varphi, \phi$ depend not only
on the spatial position $x$, but also on the shape $\bound, \domain$.
Examples are a weight function $\varphi(x,\bound) = \varphi(x,\normal)$
that depends on the normal of the surface $\bound$
\cite{Faugeras-Keriven-98,Kimmel-Bruckstein-03},
or a weight function $\phi(x,\domain)=\phi(x,u(\domain))$
that depends on the solution $u$ of a PDE defined on
$\domain$ \cite{Chan-Vese-01,Dogan-Morin-Nochetto-08,Hintermueller-Ring-04}.
To handle the computation of the shape derivatives of such
energies we need to take care of the derivative of
$\varphi,\phi$ with respect to the shape $\bound,\domain$.
For this we recall the notions of {\it material derivative} and {\it
shape derivative}.
\begin{definition}[\textrm{\cite[Prop.2.71]{Sokolowski}}]
\label{D:material-deriv}
The {\it material derivative} $\dot{\varphi}(\bound;\vV)$
of $\varphi(\bound)$ at $\bound$ in direction $\vV$
is defined as follows
\begin{equation}
  \dot{\varphi}(\bound;\vV)=
  \lim_{t\to 0}\frac{1}{t}\big( \varphi(x(t,\cdot),\bound_t)
                                -\varphi(\cdot,\bound_0)  \big),
\end{equation}
where the mapping $ x(t,\cdot)$ is defined as in (\ref{ode:mapping}).
A similar definition holds for domain functions $\phi(\domain)$.
\end{definition}
\begin{definition}[\textrm{\cite[Def. 2.85, Def 2.88]{Sokolowski}}]
\label{D:shape-deriv}
The {\it shape derivative} $\phi(\domain)$ at $\domain$
in the direction $\vV$ is defined to be
\begin{equation}
  \phi'(\domain;\vV)=\dot{\phi}(\domain;\vV)-\nabla \phi\cdot\vV.
\end{equation}
Accordingly, for surface functions $\varphi(\bound)$,
the shape derivative is defined to be
\begin{equation}
  \varphi'(\bound;\vV)=\dot{\varphi}(\bound;\vV)-
  \nabla_\bound \varphi\cdot\vV\vert_{\bound}.
\end{equation}
\end{definition}
The shape derivative concept enables us to compute the change
in the shape dependent quantities, such as the normal $\normal$
and the mean curvature $\kappa$, with respect to deformations
of the shape by given velocity fields.
\begin{lemma}\label{L:deriv-geom}
The shape derivatives of the normal $\normal$ and the mean curvature
$\kappa$ of a surface $\bound$ of class $C^2$ with respect to velocity
$\vV \in C^2$ are given by
\begin{align}
\normal' &= \normal'(\bound;\vV) = -\nablab V \label{E:shape-deriv-normal}, \\
\kappa' &= \kappa'(\bound;\vV) = -\Delta_\bound V, \label{E:shape-deriv-curv}
\end{align}
where $V=\vV\cdot\normal$ is the normal component of the velocity.
Moreover, the shape derivative of the tangential gradient of
a function $u$ of class $C^1$ defined in (or extended properly to)
a tubular neighborhood $U$ of the surface $\bound$ is
\begin{equation}\label{E:deriv-of-tangential-grad}
(\nablab u)' = \nablab u' + \nablab u\cdot\nablab V \normal
+ \frac{\partial u}{\partial\normal}\nablab V.
\end{equation}
\end{lemma}
\begin{proof}
The derivations for $\normal'$ and $\kappa'$ can be found in
\cite[Sect. 3]{Hintermueller-Ring-03}. Once we have the result for $\normal'$,
we can proceed with the following
\begin{align*}
(\nablab u)' &= (\nabla u - \nabla u\cdot\normal\normal)'
= \nabla u' - \nabla u'\cdot\normal\normal
- \nabla u\cdot\normal'\normal - \nabla u\cdot\normal\normal' \\
&= \nablab u' + \nabla u\cdot\nablab V\normal + \nabla u\cdot\normal\nablab V \\
&= \nablab u' + \nablab u\cdot\nablab V\normal
+ \frac{\partial u}{\partial\normal}\nablab V.
\end{align*}

\end{proof}

Now we state the shape derivatives of the general shape energies
~\eqref{E:general-energies}.
\begin{theorem}[\textrm{\cite[Sect. 2.31, 2.33]{Sokolowski}}]
\label{T:first-shape-deriv}
Let $\phi=\phi(\domain)$ be given so that the
material derivative $\dot{\phi}(\domain;\vV)$ and the shape
derivative $\phi'(\domain;\vV)$ exist. Then, the
shape energy $J(\domain)$ in (\ref{E:general-energies}) is  shape
differentiable and we have
\begin{equation}\label{E:deriv-domain}
  dJ(\domain;\vV)=\intd \phi'(\domain;\vV)dx + \intb \phi V dS.
\end{equation}
For surface functions $\varphi(\bound)$, the shape derivative of
$J(\bound)$ in (\ref{E:general-energies}) is given by
\begin{equation}
  dJ(\bound;\vV)=\intb \varphi'(\bound;\vV)dS+\intb \kappa \varphi V dS,
\end{equation}
whereas if $\varphi(\cdot,\bound)=\psi(\cdot,\domain)\vert_{\Gamma}$,
then we obtain
 \begin{equation}\label{E:deriv-bound}
  dJ(\bound;\vV)=\intb \psi'(\domain;\vV)\vert_{\Gamma}dS
+\intb \left(\frac{\partial\psi}{\partial\normal}+\kappa \psi \right) V dS.
\end{equation}
\end{theorem}

We conclude this section with a Riesz representation theorem,
the Hadamard-Zol{\'e}sio Theorem.
\begin{theorem}[\textrm{\cite[Sect 2.11 and Th. 2.27]{Sokolowski}}]
\label{T:hadamard}
The shape derivative of a surface or domain energy always has
a representation of the form
 \begin{equation}
   dJ(\bound;\vV)=\langle G, V \rangle_{\bound},
 \end{equation}
 where we denote by $\langle \cdot, \cdot \rangle_{\bound}$ a suitable duality
pairing on $\bound$; that is, the shape derivative is concentrated on
$\bound$.
\end{theorem}

Let us point out that an implication of this theorem is that the shape
derivative $dJ(\bound;\vV)$ depends only on $V = \vV \cdot \normal$, the
normal component of the velocity. \emph{For this reason, we will use
$V$ in our derivations from now on and assume a normal extension
when we need the velocity extended to a neighborhood of the surface. Hence,
without loss of generality, we have the following assumptions on $V$ and
$\vV$}
\begin{equation}\label{E:vel-assumptions}
\vV = V \normal, \qquad 
\frac{\partial V}{\partial\normal} = 0 \quad \textrm{ on } \bound.
\end{equation}

Deriving the first shape derivatives using only scalar velocities $V$,
thus normal velocities~\eqref{E:vel-assumptions}, may at first appear
as a restriction, since one is not always able to work with scalar velocities
$V$ and may need to use arbitrary vector velocities $\vV$. Because of
Theorem~\ref{T:hadamard}, this turns out not to be a problem; the first
shape derivatives computed with scalar velocities $V$ will be the same
as those computed with arbitrary $\vV$ with $V=\vV\cdot\normal$. 
Moreover, in the case of second shape derivatives introduced in the
next subsection, the formulas obtained with scalar velocities $V$
are the same around critical shapes as those obtained with arbitrary
vector velocities $\vV$ (given $V=\vV\cdot\normal$)
\cite{Bucur-Zolesio-97,Novruzi-Pierre-02}.

\subsection{The Second Shape Derivative}\label{S:second-deriv}

We continue to use the scalar velocity fields $V, W$ (corresponding
to vector velocity fields $\vV, \vW$ by \eqref{E:vel-assumptions})
to perturb $\bound, \domain$ and we define the second shape derivative
as follows
\begin{equation}
  d^2J(\bound;V,W)= d\left(dJ(\bound;V)\right)(\bound;W), \quad
  d^2J(\domain;V,W)= d\left(dJ(\domain;V)\right)(\domain;W).
\end{equation}
The second shape derivatives of functions $\phi(\domain)$,
$\varphi(\bound)$ can be defined similarly based on
Definition~\ref{D:shape-deriv}. Now we can use this
to compute the second shape derivative of the domain
and the surface energies. We give these results below.
\begin{lemma}[\textrm{\cite[Sect. 5]{Hintermueller-Ring-03}}]\label{L:hess-domain}
Let $\phi\in W^{2,1}(\RR^d)$ and $\bound$ be of class
$C^2$. Then the second shape derivative of the energy
\begin{equation}\label{E:func-domain-in-L1}
  J(\domain)=\int_{\domain}\phi dx
\end{equation}
at $\domain$ with respect to scalar velocity fields $V$, $W$ is given
by
\begin{equation}
  d^2J(\domain;V,W) = \intb \left( \partial_\normal \phi
+ \kappa \phi \right) V W dS.
\end{equation}
\end{lemma}
\begin{lemma}[\textrm{\cite[Sect. 5]{Hintermueller-Ring-03}}]
\label{L:hess-bound}
Let $\psi\in W^{3,1}(\RR^d)$ and $\bound$ be of class
$C^2$. Then the second shape derivative of the energy
\begin{equation}\label{E:func-bound-in-L2}
  J(\bound)=\intb\psi dS
\end{equation}
at $\bound$ with respect to scalar velocity fields $V$, $W$ is given by
 \[
d^2J(\bound;V,W) = \intb \left( \psi \nablab V \cdot \nablab W
+ \left( \partial_{\normal\normal} \psi + 2\kappa \partial_\normal \psi
+ (\kappa^2 - \sum \kappa_i^2)\psi \right) V W  \right) dS.
 \]
\end{lemma}

We employ Lemmas \ref{L:hess-domain} and \ref{L:hess-bound} and
Definition \ref{D:shape-deriv} to state the second shape derivative for
more general energies~\eqref{E:general-energies}, in which the weight
functions also depend on the shape. The assumptions
\eqref{E:vel-assumptions} are crucial to obtaining the following result.

\begin{theorem}\label{T:second-shape-deriv}
Let $\phi=\phi(x,\domain)$ be given so that the first and the second shape
derivatives $\phi'(\domain;V)$, $\phi''(\domain;V,W)$ exist. Then, the
second shape derivative of the domain energies in
(\ref{E:general-energies}) is given by
\begin{equation}\label{E:hess-domain}
d^2J(\domain;V,W)=\intd \phi''dx +
\intb \left(\phi'_W V+ \phi'_V W \right)dS
+ \intb \left( \partial_\normal \phi + \kappa \phi \right) V W dS
\end{equation}
where $\phi'_V = \phi'(\domain;V)$, $\phi'' = \phi''(\domain;V,W)$.
For surface functions $\varphi(\bound)$ in (\ref{E:general-energies})
with $\varphi(\cdot,\bound)=\psi(\cdot,\domain)\vert_{\bound}$ we obtain
\begin{eqnarray}
d^2J(\bound;V,W)&=&\intb \psi'' dS
+ \intb \left( \left(\partial_\normal \psi'_W + \kappa \psi'_W \right) V
+ \left(\partial_\normal \psi'_V + \kappa \psi'_V \right) W \right)dS
\label{E:hess-bound} \\
&+& \intb \left( \psi \nablab V \cdot \nablab W
+ \left( \partial_{\normal\normal} \psi + 2\kappa \partial_\normal \psi
+ (\kappa^2 - \sum \kappa_i^2)\psi \right) V W  \right) dS. \nonumber
\end{eqnarray}
where $\psi'_V = \psi'(\bound;V)\vert_\bound$,
$\psi''_{V,W} =\psi''(\bound;V,W)\vert_\bound$.
\end{theorem}
\begin{proof}
For $J(\domain) = \intd \phi(x,\domain) dx$, the first shape derivative at
$\domain$ in direction $V$ is
\[
\begin{array}{lcccc}
dJ(\domain;V) &=& \underbrace{\intd \phi'(\domain;V) dx}
& + &\underbrace{\intb \phi V dS}. \\
 & &  J_1 &  & J_2
\end{array}
\]
To obtain the second shape derivative, we take the derivatives of $J_1$
and $J_2$.
\begin{align*}
dJ_1(\domain;W) &= \intd \phi''(\domain;V,W) dx
+ \intb \phi'(\domain;V)W dS, \\
dJ_2(\domain;W) &= \intb \phi'(\domain;W)V dS
+ \intb (\partial_\normal \phi V + \phi V \kappa) W dS.
\end{align*}
Plugging in $d^2J(\domain;V,W)$ 
and reorganizing the terms yields the results. We obtain the second shape
derivative for the surface energy similarly.
\end{proof}

More details on the structure of the second shape derivative can be found
in \cite{Bucur-Zolesio-97,Delfour-Zolesio-01,Novruzi-Pierre-02}. We should reemphasize
that the use of assumptions \eqref{E:vel-assumptions} has allowed us to
derive the simpler forms in Theorem~\ref{T:second-shape-deriv}, compared
to these references, which investigate the second shape derivative with
more general vector velocities or perturbations. For vector velocities
with tangential components, the second shape derivative includes
additional terms, but these terms disappear at critical shapes.
Thus the formula obtained with scalar velocities is sufficient
to characterize the second shape derivative at critical shapes,
because in this situation, the second shape derivative depends only
on the normal components of the vector velocity \cite{Bucur-Zolesio-97,Novruzi-Pierre-02}.

\section{Shape Derivatives of the Model Energies}
\label{S:shape-derivs-of-energies}

We compute the first and second shape derivatives
for each shape energy class introduced in Section~\ref{S:shape-energies}.

\subsection{Minimal Surface Energies}\label{S:surface-energies}

We compute the first and the second shape derivatives of
surface energies with weight functions that depend only on the local
geometry at each point of the surface. We start with
an isotropic energy, the Geodesic Active Contour Model,
then consider anisotropic, i.e.\! normal-dependent energies;
we describe the case of a curvature-dependent weight as well.

{\bf Isotropic surface energies.}
The archetype shape energy in image processing is the
Geodesic Active Contour model~\cite{Caselles-Kimmel-Sapiro-97,Caselles-Kimmel-Sapiro-97-3d}:
\begin{equation}\label{E:gac-energy}
J(\bound) := \intb g(x)dS + \gamma \intd g(x) dx,
\end{equation}
where $\bound$ is a surface in $\RR^d$ and $\domain$
is the domain enclosed by $\bound$.
The first and second shape derivatives directly follow from
Lemmas~\ref{L:deriv-domain}, \ref{L:deriv-bound}, \ref{L:hess-domain}, \ref{L:hess-bound}
and the derivations can be found in \cite{Hintermueller-Ring-03}.
\begin{proposition}
The first shape derivative of the energy~\eqref{E:gac-energy}
at $\bound$ with respect to velocity $V$ is given by
\[
dJ(\bound;V)= \intb \Big((\kappa+\gamma)g(x)
  + \partial_\normal g(x)\Big) V dS.
\]
The second shape derivative of \eqref{E:gac-energy} with respect
to velocities $V, W$ is given by
\begin{equation*}
d^2J(\bound;V,W) = \intb \left(g \nablab V \cdot \nablab W
+ \left( \partial_{\normal\normal} g
+ (2\kappa + \gamma) \partial_\normal g
+ (\kappa^2 - \sum \kappa_i^2 + 2\gamma\kappa) g \right) V W \right)dS.
\end{equation*}
\end{proposition}


{\bf Anisotropic surface energies.} Now we compute the first
and the second shape derivatives of the anisotropic surface energy
\begin{equation}\label{E:anisotropic-energy}
J(\bound) := \intb g(x,\normal)dS.
\end{equation}
The weight function $g(x,\normal)$ depends on the orientation
of the normal of the surface. 
We assume the derivatives $g_x, g_{xx}$ with respect to the first
argument $x$ and the derivatives $g_y, g_{yy}$ with respect to the
second argument $\normal$, also the mixed derivatives $g_{xy}, g_{yx}$
are well-defined.
Applications of \eqref{E:anisotropic-energy}
are in image segmentation \cite{Kimmel-Bruckstein-03} and
in multiview stereo reconstruction \cite{Faugeras-Keriven-98,Kolev-Pock-Cremers-10}.
\begin{proposition}
The first shape derivative of the anisotropic surface energy
\eqref{E:anisotropic-energy} at $\bound$ with respect to velocity $V$
is given by
\begin{equation}\label{E:aniso-1st-deriv}
dJ(\bound;V)= {\displaystyle \intb\left( \kappa g + \partial_\normal g \right)
V - g_y \cdot \nablab V dS }
= {\displaystyle \intb \left(\kappa g + \partial_\normal g
+ \divb(g_y)_\bound \right) V dS, }
\end{equation}
where $g_y$ is the derivative of
$g(x,\normal)$ with respect to its second variable.
\end{proposition}
\begin{proof}
To derive the first derivative, we use
Theorem~\ref{T:first-shape-deriv} with $\psi = g(x,\normal)$. Then
\[
\psi'(\bound;V) = g_y \cdot \normal' = - g_y \cdot \nablab V
\]
using \eqref{E:shape-deriv-normal}.
We also need to compute the normal derivative of $g(x,\normal)$.
We have $\partial_\normal \psi = \partial_\normal (g(x,\normal))
= \partial_\normal g + g_y^T\partial_\normal \normal =
\partial_\normal g$, because $\partial_\normal \normal = 0$
by Equation~\eqref{E:normal-deriv-curv} (assuming extension by
signed distance function~\eqref{E:signed-dist},~\eqref{E:geom-signed-dist}).\\
We then substitute $\psi'$ and $\partial_\normal\psi$ in
\eqref{E:deriv-bound} and obtain
\[
dJ(\bound;V)= \intb\left( \kappa g + \partial_\normal g\right)
V -  g_y \cdot \nablab V  dS .
\]
We can apply the tangential Green's formula
\eqref{E:greens-formula} to the last term of the integral,
and use the identity $\divb(\vec{\omega})_\bound = \divb(\vec{\omega})
- \kappa \vec{\omega} \cdot \normal$ \cite[Chap. 8]{Delfour-Zolesio-01}
to obtain \eqref{E:aniso-1st-deriv}.
\end{proof}
\begin{proposition}\label{P:hess-anisotropic}
The second shape derivative of the anisotropic surface
energy~\eqref{E:anisotropic-energy} at $\bound$ with respect
to velocities $V,W$ is given by
\begin{align*}
d^2J(\bound;V,W) =& \intb \nablab V \cdot \left( (g-g_y \cdot \normal)Id +
g_{yy}\right) \cdot \nablab W  dS \\
&+ \intb \left(\partial_{\normal\normal}g+2\kappa \partial_\normal g
+ (\kappa^2 - \sum \kappa_i^2) g \right)VW dS \\
&- \intb (\kappa g_y -g_y^T\nablab\normal + \normal^T g_{xy}^T)
\cdot ( \nablab W \, V + \nablab V \, W ) dS .
\end{align*}
where $Id$ is the $d\times d$ identity matrix.
\end{proposition}
\begin{proof}
Use Theorem~\ref{T:second-shape-deriv} with $\psi = g(x,\normal)$.
Then $\partial_\normal\psi = \partial_\normal g, \
\partial_{\normal\normal} \psi = \partial_{\normal\normal} g,$
and
\begin{align*}
\psi'(\bound;V) &= - g_y \cdot \nablab V 
\\
\psi''(\bound;V,W) &= - \left( g_y \right)' \cdot \nablab V
- g_y \cdot \left( \nablab V \right)'
\qquad (\mathrm{use} \eqref{E:deriv-of-tangential-grad}, \eqref{E:vel-assumptions})
\\
&= \nablab V \cdot g_{yy} \cdot \nablab W
- g_y \cdot \normal \, \nablab V \cdot \nablab W 
\\
&= \nablab V \cdot \left( g_{yy} - g_y
\cdot \normal \, Id \right) \cdot \nablab W 
\\
\partial_\normal \psi'(\bound;V)
&= - \partial_\normal g_y \cdot \nablab V
- g_y \cdot \partial_\normal \nablab V  
\\
&= -(g_{yx} \normal) \cdot \nablab V - g_y \cdot (-\nablab\normal\nablab V) 
\\
&= (g_y^T \nablab\normal - \normal^T g_{yx}^T) \nablab V,
\end{align*}
since $\partial_\normal \nablab V = -\nablab\normal \nablab V$
by Lemma~\ref{L:deriv-surf-f-u} and assumptions~(\ref{E:vel-assumptions}).
The function $g_{yx}$ is the Hessian obtained by differentiating
with respect to the second variable $\normal$,
and to the first variable $x$. Now we substitute
the derivatives of $\psi$ in \eqref{E:hess-bound} to obtain
\begin{align*}
d^2J(\bound;V,W)=&\intb \nablab V \cdot \left( g_{yy} - g_y
\cdot \normal \, Id \right) \cdot \nablab W dS \\
&- \intb \left( (\normal^T g_{yx}^T - g_y^T\nablab\normal) \cdot \nablab W
+ \kappa g_y \cdot \nablab W \right) V \\
&- \intb \left( (\normal^T g_{yx}^T - g_y^T\nablab\normal) \cdot \nablab V
+ \kappa g_y \cdot \nablab V \right) W dS \\
&+ \intb \left( g \nablab V \cdot \nablab W
+ \left( \partial_{\normal\normal} g + 2\kappa \partial_\normal g
+ (\kappa^2 - \sum \kappa_i^2)g \right) V W  \right) dS.
\end{align*}
Reorganizing the terms yields the result.
\end{proof}


{\bf Curvature-dependent Energies.}
Next we compute the shape derivative of the curvature dependent
surface energy
\begin{equation}\label{E:curvature-energy}
J(\bound) := \intb g(x,\kappa)dS.
\end{equation}
The weight function $g(x,\kappa)$ depends on the mean curvature
of the surface. 
We assume the derivative $g_x$ with respect to the first
argument $x$ and the derivatives $g_z, g_{zz}, g_{zzz}$ 
with respect to the second argument $\kappa$, 
also the mixed derivatives $g_{zxz}, g_{zzx}$ are well-defined.
Variants of the energy~\eqref{E:curvature-energy} have been
used to impose regularity of curves or surfaces in
shape identification problems
\cite{Clarenz-etal-04,Droske-Bertozzi-10,Sundar-etal-09}.
\begin{proposition}\label{P:deriv-gen-energy}
The first shape derivative of the curvature-dependent surface energy
\eqref{E:curvature-energy} with $g = g(x,\kappa)$ at $\bound$
with respect to velocity $\vV$ is given by
\begin{align*}
dJ(\bound;V) &= \intb  - g_z \trib V
+ \left( g \kappa - g_z \sum \kappa_i^2 + \partial_\normal g \right) V dS, \\
&= \intb \left( -\trib(g_z)
+ g \kappa - g_z \sum \kappa_i^2 + \partial_\normal g \right) V dS,
\end{align*}
where $\trib(g_z)$  denotes the total derivative
\begin{equation*}
\trib(g_z) = \trib g_z + g_{zz} \trib\kappa + g_{zzz} |\nablab\kappa|^2
+ (g_{zxz} + g_{zzx})\cdot\nablab\kappa
\end{equation*}
and $g_z$ is the derivative of $g(x,\kappa)$ with respect to its second
variable; $g_{zz}, g_{zxz}, g_{zzx}, g_{zzz}$ are defined similarly.
\end{proposition}

The proof of Proposition~\ref{P:deriv-gen-energy} for the shape
derivative of \eqref{E:curvature-energy} is given by Do\u{g}an
and Nochetto in \cite{Dogan-Nochetto-12}, where they
derive the shape derivative also for the more general surface energy
\[
J(\bound) = \intb g(x,\normal,\kappa) dS.
\]

\subsection{Shape Energies with Integrals}\label{S:energies-with-integrals}

We compute the shape derivatives for the domain and surface
energies with integrals. These energies are found in applications
where we need to aggregate properties over regions or surfaces
by integration, for example, to compute statistics of given data
\cite{Aubert-etal-03,Cremers-Rousson-Deriche-07,Paragios-Deriche-02a},
and the integrals are used as parameters of the weight function.

{\bf Energies with Domain Integrals.}
We start by computing the shape derivatives for the energy
with a weight function that depends on a single domain integral:
\begin{equation}\label{E:d-energy-w-integral}
J(\domain) = \intd g(x,I_w(\domain)) dx,
\qquad I_w(\domain) = \intd w(x) dx.
\end{equation}
We then consider an energy with a weight
that depends on multiple domain integrals:
\begin{equation}\label{E:energy-w-multiple-integrals}
J(\domain) = \intd g(x,I_{w_1},\ldots,I_{w_m}) dx, \qquad
I_{w_i} = I_{w_i}(\domain) = \intd w_i(x) dx, \ i=1,\ldots,m.
\end{equation}
We also consider the case in which the dependence on domain
integrals is recursive:
\begin{equation}\label{E:energy-w-nested-integrals}
J(\domain) = \intd g_0(x,I_{g_1}) dx,
\end{equation}
where
\begin{gather*}
I_{g_k} = I_{g_k}(\domain) = \intd g_k(x,I_{g_{k+1}}) dx,
\quad k=1,\ldots,m-1, \\
I_{g_m} = I_{g_m}(\domain) = \intd g_m(x) dx.
\end{gather*}
The first shape derivatives of \eqref{E:d-energy-w-integral},
\eqref{E:energy-w-multiple-integrals} and \eqref{E:energy-w-nested-integrals}
(only for two levels of recursion) were computed in \cite{Aubert-etal-03}.
The second shape derivatives of \eqref{E:d-energy-w-integral},
\eqref{E:energy-w-multiple-integrals} are computed in this section.
\begin{proposition}\label{P:deriv-d-energy-w-integral}
The first shape derivative of the shape energy~\eqref{E:d-energy-w-integral}
at $\domain$ with respect to velocity $V$ is given by
\[
dJ(\domain;V) = \intb \left( g(x,I_w(\domain))
+ I_{g_p}(\domain) w(x) \right) V dS,
\]
The second shape derivative of \eqref{E:d-energy-w-integral}
with respect to velocities $V, W$ is given by
\begin{align*}
d^2J(\domain;V,W) =& \intb \left( \partial_\normal g + I_{g_p} \partial_\normal w
+ \kappa(g + I_{g_p} w) \right) V W dS \\
&+ \intb g_p V dS \intb w W dS + \intb w V dS \intb g_p W dS \\
&+ I_{g_{pp}} \intb w V dS \intb w W dS,
\end{align*}
We use the short notation for the functions $g = g(x,I_w(\domain))$,
$w = w(x)$, and $g_p, g_{pp}$ denote the derivatives of
$g(x,I_w(\domain))$ with respect to its second variable.
$I_{g_p}, I_{g_{pp}}$ are the integrals of $g_p, g_{pp}$
over the domain $\domain$.
\end{proposition}
\begin{proof}
Let $\phi(x,\domain) = g(x,I_w(\domain))$ in
Theorem \ref{T:first-shape-deriv},
and calculate $\phi'(\domain;V)$.
\[
\phi'(\domain;V) = g_p(x,I_w) I_w' = g_p(x,I_w) \intb w(y) V dS(y).
\]
We have used the fact that Lemma~\ref{L:deriv-domain} applies to $I_w$,
namely we have $I_w' = dI_w(\domain;V)$.
We substitute $\phi'(\domain;V)$ in the general form \eqref{E:deriv-domain}
\[
dJ(\domain;V) = \intd g_p(x,I_w) \left( \intb w(y) V dS(y) \right) dx
+ \intb g(x,I_w) V dS.
\]
Since $I_{g_p} = \intd g_p(x,I_w) dx$, we can exchange the order
of integration and obtain
\[
dJ(\domain;V) = \intb \left( g(x,I_w) + I_{g_p} w(x) \right) V dS.
\]
To compute the second shape derivative, we use
Theorem~\ref{T:second-shape-deriv}. We need
$\phi''=\phi''(\domain;V,W)$, which we compute by Lemma~\ref{L:hess-domain}
\begin{align*}
\phi'' &= (g_p)' \intb w(y) V dS
+ g_p \left( \intb w(y) V dS \right)'  \\
&= g_{pp} \intb w(z) V dS(z) \intb w(y) W dS(y)
+ g_p \intd \left( \partial_\normal w(y)
+ \kappa w(y) \right) V W dS.
\end{align*}
Substitute in \eqref{E:hess-domain}
\begin{align*}
d^2J(\domain;V,W) =& \intd g_{pp} dx \intb w V dS \intb w W dS
+ \intb g_p dx \intd \left( \partial_\normal w + \kappa w \right) V W dS \\
&+ \intb g_p V dS \intb w W dS + \intb g_p W dS \intb w V dS \\
&+ \intb ( \partial_\normal g + \kappa g ) V W dS.
\end{align*}
Reorganizing the various terms yields the result.
\end{proof}
\begin{proposition}
The first shape derivative of the energy~\eqref{E:energy-w-multiple-integrals}
at $\domain$ with respect to velocity $V$ is given by
\[
dJ(\domain;V) = \mathlarger{\intb} \left(g(x,I_{w_1},\ldots,I_{w_m}) +
\sum^m_{i=1} I_{g_{p_i}} w_i(x) \right) V dS.
\]
The second shape derivative of the energy~\eqref{E:energy-w-multiple-integrals}
at $\domain$ with respect to velocity $V, W$ is 
\begin{align*}
d^2J(\domain;V,W) =& \mathlarger{\intb}
\left( \partial_\normal g +
\sum^m_{i=1} I_{g_{p_i}} \partial_\normal w_i
+ \kappa(g + \sum^m_{i=1} I_{g_{p_i}} w_i) \right) V W dS \\
&+ \sum^m_{i=1} \left( \intb g_{p_i} V dS \intb w_i W dS
+ \intb w_i V dS \intb g_{p_i} W dS \right) \\
&+ \sum^m_{i,j=1} I_{g_{p_ip_j}} \intb w_i V dS \intb w_j W dS,
\end{align*}
\end{proposition}
\begin{proof}
The proof is essentially the same as that of
Proposition~\ref{P:deriv-d-energy-w-integral}, except that
we need to keep track of indices and corresponding
terms $w_i, I_{w_i}, g_{p_i}, I_{g_{p_i}},
g_{p_ip_j}, I_{g_{p_ip_j}}$.

\end{proof}

\begin{proposition}
The first shape derivative of the energy~\eqref{E:energy-w-nested-integrals}
at $\domain$ with respect to velocity $V$ is given by
\begin{equation*}
dJ(\domain;V) = \mathlarger{\intb} \left( g_0(x,I_{g_1})
+ \sum_{j=1}^{m-1} g_j(x,I_{g_{j+1}}) \prod_{i=0}^j I_{g_{i,p}} \right) V dS,
\end{equation*}
where $g_{i,p}$ denotes the derivative of $g_i(x,I_{g_{i+1}})$
with respect to its second argument.
\end{proposition}
\begin{proof}
Note that the shape derivative of $I_{g_m} = I_{g_m}(\domain)$
is given by
\begin{equation*}
I_{g_m}' = \intb g_m(x) V dS.
\end{equation*}
Now we compute the shape derivative of $I_{g_k} = I_{g_k}(\domain)$.
\begin{align*}
I_{g_k}' = &\intd g_{k,p}(x,I_{g_{k+1}}) I_{g_{k+1}}' dx
+ \intb g_k(x,I_{g_{k+1}}) V dS
= I_{g_{k,p}} I_{g_{k+1}}' + \intb g_k V dS \\
= &I_{g_{k,p}} \left( I_{g_{k+1,p}} I_{g_{k+2}}' + \intb g_{k+1} V dS \right)
+ \intb g_k V dS \\
= &I_{g_{k,p}} I_{g_{k+1,p}} I_{g_{k+2}}' + I_{g_{k,p}} \intb g_{k+1} V dS
+ \intb g_k V dS \\
= &I_{g_{k,p}} I_{g_{k+1,p}} \ldots I_{g_m}'
+ I_{g_{k,p}} \ldots I_{g_{m-2,p}} \intb g_{m-1} V dS + \ldots \\
&+ I_{g_{k,p}} \intb g_{k+1} V dS + \intb g_k V dS \\
= &I_{g_{k,p}} \ldots I_{g_{m-1,p}} \intb g_m V dS
+ I_{g_{k,p}} \ldots I_{g_{m-2,p}} \intb g_{m-1} V dS + \ldots \\
&+ I_{g_{k,p}} \intb g_{k+1} V dS + \intb g_k V dS.
\end{align*}
More concisely,
\begin{equation*}
I_{g_k}' = \intb g_k V dS + \sum^m_{j=k+1} \prod^{j-1}_{i=k} I_{g_{i,p}} \intb g_j V dS
= \mathlarger{\intb} \left( g_k + \sum^m_{j=k+1} \prod^{j-1}_{i=k} I_{g_{i,p}} g_j \right) V dS.
\end{equation*}
Then the first shape derivative of
energy~\eqref{E:energy-w-nested-integrals} is given by
$dJ(\bound;V) = I_{g_0}'$.
\end{proof}


{\bf Energies with Surface Integrals.}
Similar to the case with domain integrals, we compute
the shape derivatives for energies with weight functions
that depend on surface integrals
\begin{equation}\label{E:b-energy-w-integral}
J(\bound) = \intb g(x,I_w(\bound)) dS, \qquad I_w(\bound) = \intb w(x) dS.
\end{equation}
The cases of multiple integral parameters and nested integrals
in the weight function $g$ are straight-forward and are not
included in this paper.
\begin{proposition}
The first shape derivative of the energy~\eqref{E:b-energy-w-integral}
at $\bound$ with respect to velocity $V$ is given by
\[
dJ(\bound;V) = \intb \left( (g(x,I_w(\bound)) + I_{g_p} w(x)) \kappa
+ \partial_\normal g(x,I_w(\bound)) + I_{g_p} \partial_\normal w(x) \right) V dS.
\]
The second shape derivative of \eqref{E:b-energy-w-integral}
at $\bound$ with respect to velocities $V, W$ is given by
\begin{align*}
d^2J(\bound;V,W) =& \intb (g + I_{g_p} w) \nablab V \cdot \nablab W dS \\
&+ \intb \left( \partial_{\normal\normal} g +  I_{g_p} \partial_{\normal\normal} w
+2(\partial_\normal g + I_{g_p}\partial_\normal w)\kappa
+(\kappa^2 - \Sigma\kappa_i^2)(g + I_{g_p}w) \right) V W dS \\
&+ \intb \left( \partial_\normal g_p + g_p \kappa \right) V dS \intb \left(
\partial_\normal w + w \kappa \right) W dS \\
&+ \intb \left( \partial_\normal w + w \kappa \right) V dS \intb \left(
\partial_\normal g_p + g_p \kappa \right) W dS \\
&+ I_{g_{pp}} \intb \left( \partial_\normal w + w \kappa \right) V dS
\intb \left( \partial_\normal w + w \kappa \right) W dS.
\end{align*}
We use the short notation for the functions $g = g(x,I_w(\bound))$,
$w = w(x)$, and $g_p, g_{pp}$ denote the derivatives of
$g(x,I_w(\bound))$ with respect to its second variable.
$I_{g_p}, I_{g_{pp}}$ are the integrals of $g_p, g_{pp}$
over the surface $\bound$.
\end{proposition}
\begin{proof}
Let $\psi(x,\bound) = g(x,I_w(\bound))$ in
Theorem \ref{T:first-shape-deriv}.
By Lemma \ref{L:deriv-bound}, we have
\[
\psi' = \psi'(\bound;V)
= g_p(x,I_w) I_w' = g_p \intb (\partial_\normal w + w\kappa)V dS
\]
Substitute in the general form~\eqref{E:deriv-domain}
\[
dJ(\bound;V) = \intb g_p dS \intb (\partial_\normal w + w\kappa)V dS
+ \intb (\partial_\normal g + g \kappa) V dS.
\]
We let $I_{g_p} = \intb g_p(x,I_w) dS$, reorganize the terms, thus
obtain the first shape derivative.\\
Now we use Theorem~\ref{T:second-shape-deriv} and compute
$\partial_\normal \psi = \partial_\normal g, \
\partial_{\normal\normal} \psi = \partial_{\normal\normal} g,$
and
\begin{align*}
\psi_V' = \psi'(\bound;V) = &g_p \intb (\partial_\normal w + w\kappa)V dS,
\
\partial_\normal \psi_V' = \partial_\normal g_p
\intb (\partial_\normal w + w\kappa)V dS, \\
\psi'' = \psi''(\bound;V,W)
= &(g_p)'  \intb (\partial_\normal w + w\kappa)V dS +
g_p \left( \intb (\partial_\normal w + w\kappa)V dS \right)' \\
= &g_{pp}(x,I_w) \intb (\partial_\normal w + w\kappa)V dS
 \intb (\partial_\normal w + w\kappa)W dS \\
&+ g_p(x,I_w) \intb w \nablab V \cdot \nablab W dS \\
&+ g_p(x,I_w) \intb \left( \partial_{\normal\normal} w + 2\kappa \partial_\normal w
+ (\kappa^2 - \sum \kappa_i^2)w \right) V W dS, \\
\end{align*}
using Lemma \ref{L:hess-bound}. We substitute the derivatives of $\psi$ in
\eqref{E:hess-bound} and obtain
\begin{align*}
d^2J(\bound;V,W) =& \intb g_{pp} dS \intb (\partial_\normal w + w\kappa)V dS
 \intb (\partial_\normal w + w\kappa)W dS \\
&+ \intb g_p dS \intb  w \nablab V \cdot \nablab W dS \\
&+ \intb g_p dS \intb \left( \partial_{\normal\normal} w + 2\kappa \partial_\normal w
+ (\kappa^2 - \sum \kappa_i^2)w \right) V W dS \\
&+ \intb (\partial_\normal g_p + g_p\kappa) W dS
\intb (\partial_\normal w + w\kappa)V dS \\
&+ \intb (\partial_\normal g_p + g_p\kappa) V dS
\intb (\partial_\normal w + w\kappa) W dS \\
&+ \intb \left( g \nablab V \cdot \nablab W
+\left(\partial_{\normal\normal} g + 2 \kappa \partial_\normal g +
(\kappa^2 - \sum \kappa_i^2) g \right) V W \right) dS.
\end{align*}
Reorganizing the various terms yields the result.
\end{proof}
%

\subsection{Higher-Order Energies}\label{S:higher-order-energies}

In this section, we consider energies that have the form of
a double integral with a weight function
$g(x,y):\RR^d\times\RR^d\rightarrow\RR$.
These energies are used to model nonlocal interactions
between two separate spatial locations $x, y$.
Applications are found in road network detection
\cite{Rochery-Jermyn-Zerubia-06} and topology control
of curves for image segmentation
\cite{LeGuyader-Vese-08,Rocha-etal-09,Sundar-Yezzi-05}.

We introduce some notation that will simplify our derivations
of the shape derivatives. We denote by $\tilde{g}(x,y)$
the symmetricization of the function $g(x,y)$:
\begin{equation}\label{E:sym-g}
\tilde{g}(x,y) = g(x,y) + g(y,x).
\end{equation}
The derivatives of $\tilde{g}(x,y)$ are then given by
\begin{equation}\label{E:sym-g-derivs}
\begin{gathered}
\tilde{g}_x(x,y) = g_x(x,y) + g_y(y,x), \qquad
\tilde{g}_y(x,y) = g_y(x,y) + g_x(y,x), \\
\tilde{g}_{xy}(x,y) = g_{xy}(x,y) + g_{yx}(y,x),
\quad \ldots, \mathrm{\ and \ so \ on.}
\end{gathered}
\end{equation}
We also introduce the domain integral of $\tilde{g}(x,y)$
and its derivative
\begin{equation}\label{E:sym-g-d-integral}
G(x,\domain) = \intd \tilde{g}(x,y) dy, \qquad
G_x(x,\domain) = \intd \tilde{g}_x(x,y) dy,
\end{equation}
also the surface integral and its derivatives
\begin{equation}\label{E:sym-g-b-integral}
\begin{gathered}
G(x,\bound) = \intb \tilde{g}(x,y) dS(y), \qquad
G_x(x,\bound) = \intb \tilde{g}_x(x,y) dS(y), \\
G_{xx}(x,\bound) = \intb \tilde{g}_{xx}(x,y) dS(y).
\end{gathered}
\end{equation}

{\bf Higher-Order Domain Energies.}
We consider the following higher-order domain energy
\begin{equation}\label{E:domain-higher-order}
J(\domain) = \intd\intd g(x,y) \ dy dx,
\end{equation}
and compute its first and the second shape derivatives.
\begin{proposition}\label{P:domain-higher-order-deriv1}
The first shape derivative of the higher-order domain energy
\eqref{E:domain-higher-order} at $\domain$ with respect
to velocity $V$ is
\begin{equation}\label{E:domain-higher-order-deriv1}
dJ(\domain;V) = \intb G(x,\domain) V dS = \intb \intd\tilde{g}(x,y)dy V dS,
\end{equation}
where $\tilde{g}(x,y)$ is defined by~\eqref{E:sym-g}.
\end{proposition}
\begin{proof}
We define $\phi(x,\domain) = \intd g(x,y) dy$, so that we work
with a more concise form of the energy~\eqref{E:domain-higher-order}
\begin{equation}\label{E:domain-h-o-2}
J(\domain) = \intd \phi(x,\domain) dx.
\end{equation}
and use Theorem~\ref{T:first-shape-deriv}.
We start by writing the shape derivative of $\phi$
\[
\phi'(\domain;V)(x) = \intb g(x,y) V(y) dS(y),
\]
We substitute $\phi'(\domain;V)$ in the shape derivative
of~\eqref{E:domain-h-o-2},
\begin{align*}
dJ(\domain;V) &= \intd \phi'(\domain;V) dx + \intb \phi(x) V(x) dS(x) \\
&= \intd \intb g(x,y) V(y) dS(y) dx
 + \intb \intd g(x,y) dy V(x) dS(x) \\
&= \intd \intb g(y,x) V(x) dS(x) dy
 + \intb \intd g(x,y) dy V(x) dS(x) \\
&= \intd \left(\intb (g(x,y) + g(y,x)) dy\right) V(x) dS(x).
\end{align*}
and we let $\tilde{g}(x,y) = g(x,y)+g(y,x)$ to obtain the result
\eqref{E:domain-higher-order-deriv1}.
\end{proof}

\begin{remark}
Using the same strategy as in the proof of
Proposition~\ref{P:domain-higher-order-deriv1}, we can write
the shape derivatives of even higher-order energies, for example,
\[
J(\domain) = \intd\intd\intd g(x,y,z) dz dy dx.
\]
We set $\phi(x,\domain)=\intd\intd g(x,y,z) dz dy$ and compute its
shape derivative
\[
\phi'(\domain;V)(x)
= \intb \left( \intb (g(x,y,z) + g(x,z,y) dz \right) V(y) dS(y),
\]
using Proposition~\ref{E:domain-higher-order-deriv1}.
Again we use the formula for the domain shape derivative
in Theorem~\ref{T:first-shape-deriv} and substitute the current
values of $\phi(x,\domain)$ and $\phi'(V;\domain)$:
\begin{align*}
dJ(\domain;V) = &\intb \left( \intd\intd (g(x,y,z)
+ g(x,z,y)) dz dx \right) V(y) dS(y) \\
&+ \intb\intd\intd g(x,y,z) dz dy V(x) dS(x) \\
= &\intb \left(\intd\intd (g(x,y,z) + g(y,x,z)
+ g(y,z,x))dz dy \right) V(x) dS(x).
\end{align*}
\end{remark}

\begin{proposition}\label{P:domain-higher-order-deriv2}
The second shape derivative of the higher-order domain
energy~\eqref{E:domain-higher-order} at $\domain$
with respect to velocities $V,W$ is
\begin{equation}\label{E:domain-higher-order-deriv2}
\begin{aligned}
d^2J(\domain;V,W) = &\intb\intb \tilde{g}(x,y) W(y) dS(y) V(x) dS(x) \\
&+ \intb \left( \kappa(x) \intd \tilde{g}(x,y)dy
+ \normal(x)\cdot \intd \tilde{g}_x(x,y)dy \right) V W dS, \\
= &\intb\intb \tilde{g}(x,y) W(y) dS(y) V(x) dS(x) \\
&+ \intb \left( \kappa(x) G(x,\domain)
+ \normal(x)\cdot G_x(x,\domain) \right) V W dS,
\end{aligned}
\end{equation}
where $\tilde{g}(x,y), \tilde{g}_x(x,y), G(x,\domain),
G_x(x,\domain)$ are defined by \eqref{E:sym-g},
\eqref{E:sym-g-derivs}, \eqref{E:sym-g-d-integral}.
\end{proposition}
\begin{proof}
The second shape derivative is computed using Theorem~\ref{T:second-shape-deriv}.
We define $\phi(x,\domain)= \intd g(x,y) dy$, compute the derivatives
$\phi'(\domain;V), \phi''(\domain;V,W), \partial_\normal\phi$ and
substitute in formula~\eqref{E:hess-domain}.
We have $\partial_\normal \phi(x,\domain) = \normal(x)\cdot\intb g_x(x,y)\, dy,$
also
\begin{align*}
\phi'(x) &= \phi'(\domain;V)(x) = \intb g(x,y) V(y) dS(y), \\
\phi''(x) &= \phi''(\domain;V,W)(x)
= \intb \left(\kappa g(x,y)V(y) + \partial_\normal\cdot (g(x,y)V(y))\right) W(y) dS(y) \\
&= \intb \left(\kappa(y)g(x,y) + \normal(y)\cdot g_y(x,y)\right) V W dS(y).
 \quad (\partial_\normal V = 0 \ \mathrm{by \ eqn }~\eqref{E:vel-assumptions})
\end{align*}
Then the second shape derivative is given by
\begin{align*}
d^2J(\domain;V,W)=&\intd \phi''dx
+ \intb \left( \kappa \phi + \partial_\normal \phi \right) V W dS
+ \intb \left(\phi'_W V+ \phi'_V W \right)dS
\\
=&\intd \intb \left(g(x,y)\kappa(y)
+ \normal(y)\cdot g_y(x,y)\right) V W dS(y) dx \\
&+ \intb \left( \kappa(x)\intd g(x,y)dy
+ \normal(x)\cdot\intd g_x(x,y)dy \right) V W dS(x) \\
&+ \intb \left( \intb g(x,y)W(y) dS(y) V(x)
+ \intb g(x,y)V(y) dS(y) W(x)\right) dS(x).
\end{align*}
We change variables in the integrals, for example
we let
\[
\intd\intb g(x,y) K(y) V(y) W(y) dS(y) dx
= \intb \kappa(x) \intd g(y,x) dy V(x) W(x) dS(x),
\]
and reorganize the terms in $d^2J(\domain;V,W)$ and obtain
\begin{align*}
d^2J(\bound;V,W)=&\intb \left(\kappa\intd (g(x,y)+ g(y,x))dy
+ \normal\cdot\intd(g_x(x,y)+ g_y(y,x))dy \right) V W dS(x) \\
&+ \intb \intb(g(x,y)+g(y,x))W(y) dS(y) V(x) dS(x) \\
= &\intb \left(\kappa(x) \intd \tilde{g}(x,y) dy
+ \normal(x)\cdot\intd \tilde{g}_x(x,y) \right) V(x) W(x) dS(x).
\end{align*}
We substitute $G(x,\domain), G_x(x,\domain)$ for the integrals
of $\tilde{g}(x,y), \tilde{g}_x(x,y)$ respectively and obtain
the second shape derivative.
\end{proof}

{\bf Higher-Order Surface Energies.}
We consider the higher-order surface energy
\begin{equation}\label{E:bound-higher-order}
J(\bound) = \intb\intb g(x,y) dS(y) dS(x),
\end{equation}
and derive its first and second shape derivatives.

\begin{proposition}\label{P:bound-higher-order-deriv}
The first shape derivative of the higher order surface energy
$J(\bound)$~\eqref{E:bound-higher-order-deriv} at $\bound$
with respect to velocity $V$ is
\begin{equation}\label{E:bound-higher-order-deriv}
\begin{aligned}
dJ(\bound;V)&= \intb \left(\kappa(x) \intb \tilde{g}(x,y) dS(y)
+ \normal(x)\cdot\intb \tilde{g}_x(x,y) dS(y) \right) V(x) dS(x), \\
&= \intb \left(\kappa(x) G(x,\bound)
+\normal(x)\cdot G_x(x,\bound) \right) V(x) dS(x),
\end{aligned}
\end{equation}
where $\tilde{g}(x,y), \tilde{g}_x(x,y)$ are defined
by~\eqref{E:sym-g},~\eqref{E:sym-g-derivs}
and $G(x,\bound), G_x(x,\bound)$ by~\eqref{E:sym-g-b-integral}.
\end{proposition}
\begin{proof}
We will use Theorem~\ref{T:first-shape-deriv} to calculate the first
shape derivative of the energy~\eqref{E:bound-higher-order}.
We define $\psi(x,\bound)=\intb g(x,y) dS(y)$ (so that
$J(\bound)=\intb \psi(x,\bound) dS(x)$) and calculate
its derivatives of $\psi$ are
\[
\partial_\normal\psi = \normal(x)\cdot\intb g_x(x,y) dS(y), \qquad
\psi'(\bound;V) = \intb(g(x,y) \kappa(y)
+ g_y(x,y)\cdot\normal(y))V dS(y).
\]
to be substituted in the formula~\eqref{E:deriv-bound} for the first
shape derivative:
\begin{align*}
dJ(\bound;V) =&\intb \psi'(\bound;V) dS(x)
  + \intb(\psi\kappa + \partial_\normal\psi_x) V dS(x) \\
=&\intb\intb \left( g(x,y) \kappa(y)
  + g_y(x,y)\cdot\normal(y) \right) V(y) dS(y) dS(x) \\
&+ \intb\left( \kappa(x)\intb g(x,y) dS(y)
  + \normal(x) \cdot \intb g_x(x,y) dS(y) \right) V(x) dS(x).
\end{align*}
We exchange the variables in the first integral,
\begin{align*}
dJ(\bound;V) =&\intb\intb \left( g(y,x) \kappa(x)
  + \normal(x) \cdot g_y(y,x) \right) V(x) dS(x) dS(y) \\
&+ \intb\left( \kappa(x)\intb g(x,y) dS(y)
  + \normal(x) \cdot \intb g_x(x,y) dS(y) \right) V(x) dS(x) \\
=&\intb\left( \kappa(x)\intb (g(x,y)+g(y,x)) dS(y) \right.\\
&\quad \left.+ \normal(x) \cdot \intb(g_x(x,y)
          + g_y(y,x)) dS(y) \right) V(x) dS(x) \\
=&\intb\left( \kappa(x)\intb \tilde{g}(x,y) dS(y)
+ \normal(x) \cdot \intb\tilde{g}_x(x,y) dS(y) \right) V(x) dS(x).
\end{align*}
We also replace the integrals of $\tilde{g}(x,y), \tilde{g}_x(x,y)$
by $G(x,\bound), G_x(x,\bound)$ respectively.
\end{proof}
\begin{proposition}\label{P:bound-higher-order-hess}
The second shape derivative of the higher order surface energy
$J(\bound)$~\eqref{E:bound-higher-order-deriv} at $\bound$
with respect to velocities $V, W$ is
\begin{align*}
d^2J(\bound; & V,W) = \intb G(x,\bound)\nablab V\cdot \nablab W dS(x)
\\
&+ \intb \left( \normal^T G_{xx}(x,\bound)\,\normal
+ 2\kappa\, G_{x}(x,\bound)\cdot\normal
+ (\kappa^2 - \Sigma\kappa_i^2) G(x,\bound) \right) V W dS(x)
\\
&+ \intb \kappa \intb \tilde{g}\kappa\, W dS(y) V dS(x)
+ \intb \normal^T \intb \tilde{g}_{xy}\normal\, W dS(y) V dS(x)
\\
& + \intb \kappa \intb \tilde{g}_y \cdot\normal\, W dS(y) V dS(x)
+ \intb \normal\cdot\intb\tilde{g}_x\kappa\, W dS(y) V dS(x),
\end{align*}
where $\tilde{g}(x,y), \tilde{g}_x(x,y), \tilde{g}_y(x,y),
\tilde{g}_{xy}(x,y), G(x,\bound), G_{x}(x,\bound), G_{xx}(x,\bound)$
are defined by the formulas~\eqref{E:sym-g},\eqref{E:sym-g-derivs},~\eqref{E:sym-g-b-integral}.
\end{proposition}
\begin{proof}
We define $\psi(x,\bound)=\intb g(x,y) dS(y)$ and compute its
normal derivatives $\partial_\normal\psi, \ \partial_{\normal\normal}\psi$
and its shape derivatives $\psi_V'=\psi'(\bound;V), \ \psi''=\psi''(\bound;V,W),
\ \partial_\normal\psi_V'$:
\begin{gather*}
\partial_\normal\psi = \normal(x)\cdot\intb g_x(x,y)dS(y), \quad
\partial_{\normal\normal}\psi = \normal(x)^T\left(\intb g_x(x,y)dS(y)\right)\normal(x), \\
\psi_V' = \intb (g\kappa(y) + g_y\cdot\normal(y)) V dS(y), \quad
\partial_\normal\psi_V'
= \normal\cdot\intb (g_x\kappa(y) + g_{yx}^T\normal(y)) V dS(y), \\
\psi'' = \intb g\nablab V\cdot \nablab W dS(y)
+ \intb \left( \normal^T g_{yy} \normal + 2\kappa g_y\cdot\normal
+ (\kappa^2 - \Sigma\kappa_i^2) g \right) V W dS(y),
\end{gather*}
which we subsitute in the general formula~\eqref{E:hess-bound}
for the second shape derivative
\begin{align*}
d^2J(\bound;V,W)= &\intb \psi'' dS
+ \intb \left( \left(\partial_\normal \psi'_W + \kappa \psi'_W \right) V
+ \left(\partial_\normal \psi'_V + \kappa \psi'_V \right) W \right)dS \\
&+ \intb \left( \psi \nablab V \cdot \nablab W
+ \left( \partial_{\normal\normal} \psi + 2\kappa \partial_\normal \psi
+ (\kappa^2 - \Sigma \kappa_i^2)\psi \right) V W  \right) dS.
\end{align*}
and obtain
\begin{align*}
d^2J(\bound;V,W) = &\intb\intb g \nablab V(y) \cdot \nablab W(y) dS(y)dS(x)
\\
&+ \intb\intb \left( \normal(y)^T g_{yy}\normal(y) + 2\kappa(y) g_y\cdot\normal(y) \right. \\
& \qquad \ \left. {}+ (\kappa(y)^2 - \Sigma \kappa_i(y)^2) g \right) V(y) W(y) dS(y) dS(x)
\\
&+ \intb \left( \normal(x)\cdot \intb(g_x\kappa(y) + g_{yx}^T \normal(y))W(y)dS(y) \right. \\
& \qquad \ \left. {}+ \kappa(x) \intb(g\kappa(y) + g_y\cdot\normal(y))W(y)dS(y) \right) V(x) dS(x)
\\
&+ \intb \left( \normal(x)\cdot \intb(g_x\kappa(y) + g_{yx}^T \normal(y))V(y)dS(y) \right. \\
& \qquad \ \left. {}+ \kappa(x) \intb(g\kappa(y) + g_y\cdot\normal(y))V(y)dS(y) \right) W(x) dS(x)
\\
&+ \intb\intb g dS(y) \nablab V(x) \cdot \nablab W(x) dS(x)
\\
&+ \intb\left( \normal(x)^T \intb g_{xx} dS(y)\normal(x) + 2\kappa(x) \normal(x)\cdot\intb g_x dS(y) \right. \\
& \qquad \ \left. {}+ (\kappa(x)^2 - \Sigma \kappa_i(x)^2) \intb g dS(y) \right) V(x) W(x) dS(x).
\end{align*}
We exchange variables $x \leftrightarrow y$ in the integrals that
contain $V(y)$ and reorganize,
\begin{align*}
d^2J(\bound;V,W) = &\intb \left(\intb(g(x,y)+g(y,x))dS(y)\right) \nablab V\cdot\nablab W dS(x)
\\
&+ \intb\left\{ \normal(x)^T \left(\intb (g_{xx}(x,y)+g_{yy}(y,x)) dS(y)\right)\normal(x) \right. \\
& \qquad \ {}+ 2\kappa(x)\normal(x)\cdot \left(\intb (g_x(x,y)+g_y(y,x)) dS(y)\right) \\
& \qquad \ \left. {}+ (\kappa(x)^2 - \Sigma \kappa_i(x)^2) \intb (g(x,y)+g(y,x)) dS(y) \right\} V W dS(x)
\\
&+ \intb \kappa(x) \left(\intb (g(x,y)+g(y,x))\kappa(y) W(y)dS(y)\right) V(x) dS(x)
\\
&+ \intb \normal(x)\cdot \left(\intb (g_{xy}(x,y)+g_{yx}(y,x))\normal(y) W(y)dS(y)\right) V(x) dS(x)
\\
&+ \intb \kappa(x) \left(\intb (g_y(x,y)+g_x(y,x))\cdot\normal(y) W(y)dS(y)\right) V(x) dS(x)
\\
&+ \intb \normal(x)\cdot \left(\intb (g_x(x,y)+g_y(y,x))\kappa(y) W(y)dS(y)\right) V(x) dS(x).
\end{align*}
We replace instances of $g(x,y)$ and its derivatives by
$\tilde{g}(x,y)$ and its derivatives, as defined by~\eqref{E:sym-g}.
We also replace the integrals of $\tilde{g}, \tilde{g}_x, \tilde{g}_{xx}$
by $G(x,\bound), G_x(x,\bound), G_{xy}(x,\bound)$ respectively.
\end{proof}

\subsection{Shape Energies with PDEs}\label{S:energies-with-pdes}

In some problems, the shape energy may include a function
$u$ that is obtained by solving a PDE on the surface $\bound$
or in the domain $\domain$ enclosed by the surface $\bound$,
namely, we consider energies of the form
\[
J(\bound) = J_0(\bound,u(\bound)), \quad
\mathcal{A}_\bound (u) = f \ \mathrm{on} \ \bound,
\quad \mathrm{or} \quad
J(\domain) = J_0(\domain,u(\domain)), \quad
\mathcal{A}_\domain (u) = f \ \mathrm{in} \ \domain,
\]
where $\mathcal{A}_\bound,\mathcal{A}_\domain$
denote some differential operators.
The PDEs represented by $\mathcal{A}_\bound,\mathcal{A}_\domain$
might be in various forms
and do not seem to be of interest in image processing except for
a few specific cases relating to the Mumford-Shah functional
\cite{Brox-Cremers-09,Chan-Vese-01,Dogan-Morin-Nochetto-08,Hintermueller-Ring-04,Jin-Yezzi-Soatto-03}
(see Section~\ref{S:shape-energies}).
We will consider these cases below.
For other examples of shape energies with PDEs in areas
outside image processing, we refer to the books
\cite{Delfour-Zolesio-01,Haslinger-Makinen-03,Mohammadi-Pironneau-01,Pironneau-84}.

{\bf Energies with domain PDEs.}
We consider the generalization \eqref{E:gen-MS-energy1} of 
the Mumford-Shah energy introduced in Section~\ref{S:shape-energies}
and compute the first and second shape derivatives of 
the PDE-dependent part, namely the following domain energy,
\begin{equation}\label{E:gen-MS-energy}
J(\domain) = \intd f(x,\{u_k\}) dx + 
\frac{\mu}{2} \sum_{k=1}^m \intd |\nabla u_k|^2 dx,
\end{equation}
where the smooth approximation functions $\{u_k\}_{k=1}^m$ are 
computed from the PDEs
\begin{equation}\label{E:gen-MS-pde}
- \Delta u_l + f_{u_l}(x,\{u_k\}) = 0 \ \mathrm{in} \ \domain, \quad 
\frac{\partial u_l}{\partial \normal} = 0 \ \mathrm{on} \ \bound=\partial\domain, 
\quad l=1,\ldots,m.
\end{equation}
In \eqref{E:gen-MS-pde}, $f_{u_l}$ denotes the derivative of 
the coupled data function $f(x,\{u_k\})$
with respect to the argument $u_l$ and we assume $f$ is given such that
unique solutions of \eqref{E:gen-MS-pde} and \eqref{E:MS-u-deriv-pde} 
exist in $H^1(\domain)$. The shape derivations
for \eqref{E:gen-MS-energy} follow those of Hinterm\"uller and Ring
\cite{Hintermueller-Ring-04} for the Mumford-Shah energy~\eqref{E:MS-energy0}
\cite{Chan-Vese-01}.

\begin{proposition}\label{P:deriv-MS}
The first shape derivative of the energy~\eqref{E:gen-MS-energy}
at $\domain$ with respect to velocity $V$ is
\begin{equation}
dJ(\domain;V)= \intb \left(f(x,\{u_k\})
+ \frac{\mu}{2} \sum_k |\nablab u_k|^2 \right) V dS.
\end{equation}
The second shape derivative of \eqref{E:gen-MS-energy} 
with respect to velocities $V, W$ is
\begin{equation}\label{E:gen-MS-hess}
\begin{aligned}
d^2J(\domain;V,W) = 
&\intb \left(f\kappa + \frac{\partial f}{\partial\normal} 
+ \mu\sum_k \nablab u_k^T \left( \frac{\kappa}{2}Id 
- \nablab\normal \right) \nablab u_k \right)  V W dS \\
& +\sum_k \intb \left( f_{u_k} u_{k,W}' 
+ \mu \nablab u_k\cdot\nablab u_{k,W}' \right) V dS,
\end{aligned}
\end{equation}
where $u_{k,W}' = u_k'(\domain;W)$ is the shape derivative
of $u_k$ at $\domain$ with respect to $W$ computed from the PDEs
\begin{equation}\label{E:MS-u-deriv-pde}
-\mu \Delta u_k' + \sum_l f_{u_k u_l} u_{l,W}' 
= 0 \ \mathrm{in} \ \domain, \quad
\frac{\partial u_{k,W}'}{\partial \normal}
= \divb(V\nablab u_k) - \frac{1}{\mu} f_{u_k} V \ \mathrm{on} \ \bound.
\end{equation}
\end{proposition}
\begin{proof}
We use Theorem~\ref{T:first-shape-deriv} and let
$\phi(x,\domain) = f(x,\{u_k\}) + \frac{\mu}{2}\sum_k|\nabla u_k|^2$
so that
\[
\phi'(\domain;V) = \sum_k f_{u_k}(x,\{u_l\}) u_k' 
+ \mu \sum_k \nabla u_k \cdot \nabla u_k',
\]
where $u_k' = u_k'(\domain;V)$. From Theorem~\ref{T:first-shape-deriv}
it follows that
\begin{equation}\label{E:part-MS-intermediate}
dJ(\domain;V) = \intd \sum_k \left(f_{u_k} u_k'
+ \mu \nabla u_k \cdot \nabla u_k' \right) dx
+ \intb \left( f + \frac{\mu}{2} \sum_k |\nablab u_k|^2 \right) V dS.
\end{equation}
Note that $\nabla u_k = \nablab u_k$ on $\bound$ since 
$\frac{\partial u_k}{\partial\normal} = 0$ on $\bound$. \\
Now we investigate $u_k'(\domain;V)$. Consider the weak form of
\eqref{E:gen-MS-pde}
\begin{equation}\label{E:MS-pde-weak-form}
\int_{\domain} (\mu \nabla u_l \cdot \nabla \varphi + f_{u_l} \varphi) dx = 0,
\qquad \forall \varphi \in H^1(\domain)
\end{equation}
and take the shape derivative ($\varphi$ is shape-independent)
\[
\int_{\domain} \left(\mu \nabla u_l' \cdot \nabla \varphi 
+ \sum_k f_{u_l u_k}u_k' \varphi \right) dx
+ \intb (\mu \nabla u_l \cdot\nabla \varphi + f_{u_l} \varphi) V dS = 0.
\]
Again recall that $\frac{\partial u_l}{\partial\normal}=0$
and substitute $\nabla u|_\bound = \nablab u$,
\begin{equation}\label{E:MS-u-deriv-weak-form0}
\intd \left(\mu \nabla u_l' \cdot \nabla \varphi 
+ \sum_k f_{u_l u_k}u_k' \varphi \right) dx =
-\intb (\mu \nablab u_l \nablab \varphi + f_{u_l} \varphi) V dS.
\end{equation}
This equation has a unique solution $u_l' \in
H^1(\domain)$ and we use it as a test
function in \eqref{E:MS-pde-weak-form}. In this way we see that the first
integral in \eqref{E:part-MS-intermediate} vanishes, which leaves us with
the expression for the first shape derivative. \\
We also write the strong form of the PDE for the shape
derivative $u_l'$. For this, we first integrate the left 
hand side of~\eqref{E:MS-u-deriv-weak-form0} by parts
with tangential Green's formula~\eqref{E:greens-formula}
\[
\intd \left(\mu \nabla u_l' \cdot \nabla \varphi 
+ \sum_k f_{u_l u_k}u_k' \varphi \right) dx =
\intb \Big(\mu \divb(V \nablab u_l) - f_{u_l} V \Big) \varphi dS.
\]
Then we have the following PDE for $u_l'$
\begin{equation*}
-\mu \Delta u_l' + \sum_k f_{u_l u_k} u_k' 
= 0 \ \mathrm{in} \ \domain, \quad
\frac{\partial u_l'}{\partial \normal}
= \divb(V\nablab u_l) - \frac{1}{\mu} f_{u_l} V \ \mathrm{on} \ \bound.
\end{equation*}

To compute the second shape derivative, we now let 
$ \psi(x,\bound) = \big( f(x,\{u_k\}) + 
\frac{\mu}{2}\sum_k |\nablab u_k|^2 \big) V$
and apply Theorem~\ref{T:first-shape-deriv}.
We compute the derivatives $\frac{\partial \psi}{\partial\normal}$
and $\psi'(\bound;W)$.
\[
\frac{\partial \psi}{\partial\normal} =
\left( \frac{\partial f}{\partial\normal} 
+ \sum_k f_{u_k} \frac{\partial u_k}{\partial\normal} 
+ \frac{\mu}{2} \sum_k 
\frac{\partial}{\partial\normal} |\nablab u_k|^2 \right) V
+ \Big( \ldots \Big) \frac{\partial V}{\partial\normal}.
\]
Recall that $\frac{\partial u_k}{\partial\normal} = 0$ on $\bound$
and $\frac{\partial V}{\partial\normal} = 0$
by the assumptions~\eqref{E:vel-assumptions}.
Also we have $\frac{\partial}{\partial\normal}(\nablab u) 
= -\nablab\normal\nablab u$ by Lemma~\ref{L:deriv-surf-f-u}. 
Therefore
\[
\frac{\partial\psi}{\partial\normal} = 
\left( \frac{\partial f}{\partial\normal} 
- \mu\sum_k \nablab u_k^T \nablab\normal \nablab u_k \right) V.
\]
We compute the shape derivative $\psi' = \psi'(\bound;W)$,
\begin{equation*}
\psi' = \left( \sum_k f_{u_k} u_{k,W}' 
+ \mu \sum_k\nablab u_k\cdot (\nablab u_k)_W'\right) V 
+ \Big(\ldots \Big) V'.
\end{equation*}
By Lemma~\ref{L:deriv-geom} and 
$\frac{\partial u_k}{\partial\normal} = 0$, we have 
$(\nablab u_k)' = \nablab u_{k,W}' + \nablab u_k\cdot\nablab W\normal$
and one can trivially see
$V' = (\vV\cdot\normal)' = \vV\cdot(-\nablab W) 
= - V\normal\cdot\nablab W = 0$.
Therefore
\begin{align*}
\psi' &= \left( \sum_k f_{u_k} u_{k,W}' 
+ \mu \sum_k \left(\nablab u_k\cdot\nablab u_{k,W}'
+ \nablab u_k\cdot\normal \nablab u_k\cdot\nablab W\right) \right) V \\
& = \sum_k \left(f_{u_k} u_{k,W}' + \mu\nablab u_k\cdot\nablab u_{k,W}'\right) V.
\qquad (\nablab u_k\cdot\normal = 0)
\end{align*}
The second shape deriative is obtained by plugging 
$\psi, \frac{\partial\psi}{\partial\normal}, \psi'$
in 
\begin{equation*}
d^2J(\domain;V,W) = \intb \psi'(\bound;W) dS 
+ \intb\left( \psi\kappa + \frac{\partial\psi}{\partial\normal} \right) W dS.
\end{equation*}

\end{proof}
%


{\bf Energies with Surface PDEs.}
In this section we derive the first shape derivative of 
the PDE-dependent surface energy
\begin{equation}\label{E:part-surf-MS-energy}
J(\bound) = \frac{1}{2} \int_\bound \left( f(x,\{u_k\})
+ \frac{\mu}{2} \sum_k |\nablab u_k|^2 \right) dS,
\end{equation}
where the smooth surface functions $\{u_k\}_{k=1}^m$ are
computed from the PDE
\begin{equation}\label{E:surf-MS-pde}
-\mu \trib u_l + f_{u_l}(x,\{u_k\}) = 0 \ \textrm{on} \ \bound, \quad 
l=1,\ldots,m.
\end{equation}
In \eqref{E:surf-MS-pde}, $f_{u_l}$ denotes the derivative of 
the coupled data function $f(x,\{u_k\})$
with respect to the argument $u_l$ and we assume $f$
is given such that unique solutions of the PDEs \eqref{E:surf-MS-pde},
\eqref{E:surf-MS-u-deriv} exist in $H^1(\bound)$.
The shape energy \eqref{E:part-surf-MS-energy} can be used 
for shape identification problems, in which smooth approximations
$\{u_k\}$ of data channels or descriptors on the surface $\bound$ need 
to be estimated in addition to the surface $\bound$ itself 
(e.g.\! sterescopic segmentation \cite{Jin-Yezzi-Soatto-03}).

\begin{proposition}
The first shape derivative of the energy~\eqref{E:part-surf-MS-energy}
at $\bound$ with respect to velocity $V$ is given by
\begin{equation*}
dJ(\bound;V) = \intb \left(f \kappa
+ \frac{\partial f}{\partial\normal} 
+ \mu \sum_k \nablab u_k^T \left(\frac{\kappa}{2}Id 
   - \nablab\normal\right)\nablab u_k \right) V dS.
\end{equation*}
\end{proposition}
\begin{proof}
We take the first shape derivative 
using Theorem~\ref{T:first-shape-deriv}:
\begin{equation}\label{E:part-surf-MS-deriv-1}
\begin{aligned}
dJ(\bound;V) =& \intb\left(\sum_k f_{u_k} u_k' 
    + \mu\sum_k \nablab u_k\cdot(\nablab u_k)'\right)dS \\
&+ \intb \kappa\left(f(x,\{u_k\})
    + \frac{\mu}{2}\sum_k|\nablab u_k|^2\right) V dS \\
&+ \intb \frac{\partial}{\partial n}\left( f(x,\{u_k\}) 
    + \frac{\mu}{2}\sum_k|\nablab u_k|^2\right) V dS.
\end{aligned}
\end{equation}
Meaningful interpretation of the expression~\eqref{E:part-surf-MS-deriv-1}
requires the functions $\{u_k\}$ to be defined off the surface $\bound$,
because we need to be able to compute their full spatial gradient
and the normal derivatives.
But $\{u_k\}$ are computed with the PDE~\eqref{E:surf-MS-pde}
and are defined only on $\bound$. To be able to proceed with the derivations,
we work with smooth extensions $\{\tilde{u}_k\}$ of $\{u_k\}$ in a tubular
neighborhood $U$ of the surface $\bound$. We define the extension
$\tilde{u}_k$ such that it is constant in the normal direction,
i.e.\! $\frac{\partial \tilde{u}_k}{\partial\normal} = 0$.
To keep notation simple, we will continue to refer to the extended
function as $u_k$.

Using $\frac{\partial u_k}{\partial\normal} = 0$ and identity \eqref{E:mixed-normal-tangential-deriv}, we find
\begin{equation*}
\frac{1}{2}\frac{\partial}{\partial\normal} |\nablab u_k|^2
= -\nablab u_k^T\nablab\normal \nablab u_k, 
\qquad
\frac{\partial}{\partial n} \left( f(x,\{u_k\}) \right)
= \frac{\partial f}{\partial n}.
\end{equation*}
Then using equation~\eqref{E:deriv-of-tangential-grad} and noting
$\nablab u_k\cdot\normal=0$ and $\frac{\partial u_k}{\partial n} = 0$, 
we write
\[
\nablab u_k\cdot(\nablab u_k)' =  \nablab u_k \cdot 
\left(\nablab u_k' + \nablab u_k\cdot\nablab V \normal 
+ \frac{\partial u_k}{\partial n} \nablab V \right)
= \nablab u_k \cdot \nablab u_k'.
\]
Now we can rewrite the shape derivative
\begin{equation}\label{E:part-surf-MS-deriv-2}
\begin{aligned}
dJ(\bound;V) =& \intb\left(\sum_k f_{u_k} u_k' 
    + \mu\sum_k\nablab u_k\cdot\nablab u_k'\right)dS \\
&+ \intb \kappa\left(f + \frac{\mu}{2}\sum_k |\nablab u_k|^2\right) V dS
 + \intb \left(\frac{\partial f}{\partial n} 
    - \mu\nablab u^T \nablab\normal \nablab u\right) V dS,
\end{aligned}
\end{equation}
in which the terms containing the shape derivatives $u'$
will vanish as we will see below. To show this,
we start with the weak form of the surface PDE~\eqref{E:surf-MS-pde}
\begin{equation}\label{E:surf-MS-pde-weak-form}
\intb \left(\mu\nablab u_l\cdot\nablab\varphi + f_{u_l}\varphi \right) dS
= 0 , \qquad \forall \varphi \in H^1(\bound),
\end{equation}
also using normal extensions $\tilde{\varphi}$ of the test functions
$\varphi$ with $\frac{\partial \tilde{\varphi}}{\partial n} = 0$
(which we continue to refer to as $\varphi$).
We differentiate the two terms 
in \eqref{E:surf-MS-pde-weak-form}. Start with the second term,
\begin{align*}
\left(\intb f_{u_l} \varphi dS\right)'
= \intb \sum_k f_{u_l u_k} u_k' \varphi dS
+ \intb \left( \kappa f_{u_l} + 
\frac{\partial f_{u_l}}{\partial\normal} \right) \varphi V dS.
\end{align*}
Then the first term,
\begin{equation}\label{E:part-surf-MS-pde-deriv-1}
\begin{aligned}
\left(\intb \mu\nablab u_l\cdot\nablab\varphi dS\right)'
&= \intb \left(\mu(\nablab u_l)'\cdot\nablab\varphi
+ \mu\nablab u_l\cdot(\nablab\varphi)' \right)dS \\
& + \intb \kappa \mu\nablab u_l\cdot\nablab\varphi V dS
+ \intb \frac{\partial}{\partial\normal}
(\mu\nablab u_l\cdot\nablab\varphi) V dS.
\end{aligned}
\end{equation}
We use Lemmas~\ref{L:deriv-surf-f-u},~\ref{L:deriv-geom}
to rewrite the following terms
\begin{align*}
(\nablab u_l)'\cdot\nablab\varphi 
   &= \nablab u_l'\cdot\nablab\varphi
     + \nablab u_l\cdot\nablab V \nablab \varphi\cdot\normal
     + \frac{\partial u_l}{\partial\normal} 
       \nablab\varphi\cdot\nablab V
     = \nablab u_l'\cdot\nablab\varphi,
\\
\nablab u_l\cdot(\nablab\varphi)'
   &= \nablab u_l\cdot \nablab\varphi' = 0, 
      \qquad\qquad (\mathrm{note:} \ \nablab \varphi\cdot\normal=0, \ 
                   \frac{\partial u_l}{\partial\normal} = 0)
\\
\frac{\partial}{\partial\normal} (\nablab u_l\cdot\nablab\varphi)
   &= -\nablab\varphi^T \nablab\normal \nablab u_l
     -\nablab u_l^T \nablab\normal \nablab\varphi
    = -2 \nablab u_l^T \nablab\normal \nablab\varphi,
\end{align*}
and substitute back in~\eqref{E:part-surf-MS-pde-deriv-1}.
Then the shape derivative of weak form~\eqref{E:surf-MS-pde-weak-form}
is 
\begin{equation}\label{E:surf-MS-u-deriv}
\begin{aligned}
\intb\left(\mu\nablab u_l'\cdot\nablab\varphi \right. &
+ \left. \sum_k f_{u_l u_k} u_k' \varphi\right) dS \\
&= \intb\left( \kappa f_{u_l}\varphi 
  + \frac{\partial f_{u_l}}{\partial\normal}\varphi
  +\mu \nablab u_l^T (\kappa - 2\nablab\normal) \nablab\varphi \right) V dS.
\end{aligned}
\end{equation}
The function $f$ is given such that these coupled PDEs
have a unique solution $\{u_l'\}$ in $H^1(\domain)$.
We plug in the solutions $\{u_l'\}$ as test functions in \eqref{E:surf-MS-pde-weak-form} and find that
\[
\intb \left(\mu\nablab u_l\cdot\nablab u_l' + f_{u_l} u_l' \right) dS = 0.
\]
This removes the first term in the shape
derivative~\eqref{E:part-surf-MS-deriv-2}.
\end{proof}

\section{Gradient Descent Flows}
\label{S:grad-descent}

The main motivation for deriving the shape derivatives
$dJ(\bound;V), d^2J(\bound;V,W)$ of a given shape
energy $J(\bound)$ is to design algorithms for minimization
of the energy $J(\bound)$ and for computing the optimal
shape $\bound^*$. In this section, we briefly
review how to develop gradient descent flows, namely
energy-decreasing evolutions of the shapes, using shape
derivatives for this purpose. We refer to
\cite{Almgren-Taylor-Wang-93,Almgren-Taylor-95,Ambrosio-95,Ambrosio-Gigli-Savare-05,Dogan-etal-07,Taylor-Cahn-94}
for more information on this topic.

We note, in Theorem~\ref{T:hadamard}, that the shape
derivative has the following form
\[
dJ(\bound;V) = \intb G(\bound) V dS,
\]
where $G(\bound)$ (or alternatively $G(\domain)$) is
the shape gradient depending on the shape energy.
It is easy to see that, formally by setting $V=-G(\bound)$,
we obtain a gradient descent velocity
\begin{equation}
dJ(\bound;V) = -\intb G^2 dS \leqslant 0.
\end{equation}
The velocity $V=-G(\bound)$ is the most commonly used
gradient descent velocity for shape optimization problems
in image processing. Given a method to compute the gradient
descent velocity $V$, we can now perform
the minimization by starting from an initial surface
$\bound_0$ and updating it iteratively, recomputing
the velocity for the new shape $\bound_{k+1}$
at each step:
\begin{equation}\label{E:surface-update}
\vX_{k+1} = \vX_k + \tau_k \vV_k, \qquad
\forall \vX \in \bound_k,
\end{equation}
where $\tau_k>0$ is a step size parameter that can be
fixed or chosen by a line search algorithm. 
The vector velocity $\vV$ can be computed from the normal 
velocity $V$; for a surface $\bound$ with normal $\normal$,
a natural choice is $\vV=V\normal$
as the tangential component of the velocity $\vV$
does not change the shape of the surface.
An alternative to the explicit update~\eqref{E:surface-update}
is to embed the surface in a Eulerian representation,
such as a level set function $\varphi$,
extend the velocity $V$ off the surface, and compute
the level set evolution solving the following PDE:
$\frac{\partial\varphi}{\partial t} = V |\nabla\varphi|$
\cite{Osher-Fedkiw-03,Osher-Sethian-88,Sethian-99}.

Other gradient descent velocities than $V=-G(\bound)$ are possible
\cite{Ambrosio-Gigli-Savare-05,Charpiat-etal-07,Dogan-etal-07,Sundar-Yezzi-Mennucci-07,Sundar-etal-09,Taylor-Cahn-94}.
We can introduce a scalar product $b(\cdot,\cdot)$
associated with a Hilbert space $B(\bound)$ on the
surface $\bound$ and use it to compute a different
gradient flow by solving the following equation
\begin{equation}\label{E:general-vel}
b(V,W) = -\intb G(\bound) W dS, \qquad \forall W \in B(\bound).
\end{equation}
It is easy to verify that the solution $V$ of equation
\eqref{E:general-vel} is a gradient descent velocity;
we substitute it in the shape derivative and see that 
$dJ(\bound;V) = -b(V,V) \leqslant 0$ (as the scalar
product $b(\cdot,\cdot)$ is positive definite).

The velocity $V=-G(\bound)$ mentioned above is actually
the $L^2$ gradient descent velocity obtained by setting
the scalar product equal to the $L^2$ scalar product,
$b(V,W) = \intb V W dS$, in \eqref{E:general-vel}.
We can take advantage of other scalar products
$b(\cdot,\cdot)$ in order to obtain velocities
that improve the descent process in various ways
\cite{Charpiat-etal-07,Dogan-etal-07,Sundar-Yezzi-Mennucci-07,Sundar-etal-09}.
For example, an $H^1$ scalar product,
\[
b(V,W) = \intb \alpha(x) \nablab V\cdot\nablab W + \beta(x)V W dS,
\quad (\alpha(x),\beta(x)>0),
\]
results in smoother velocities that are advantageous
in applications of segmentation and tracking
\cite{Sundar-etal-06,Sundar-etal-09}.

Another option is to use the second shape derivative
as the basis of the scalar product, for example, set
\begin{equation}\label{E:Newton-vel}
b(V,W) = d^2J(\bound;V,W),
\end{equation}
This choice results in a Newton's method for shape
optimization. It can yield quadratic convergence
in the neighborhood of the solution.
However, direct use of the scalar product
\eqref{E:Newton-vel} may not always be possible, because
the second shape derivative $d^2J(\bound;\vV,\vW)$
may not always satisfy the properties of a scalar product,
for example, it may not be positive definite.
In this case, one can still design a scalar product
$b(\cdot,\cdot)$ based on the second shape derivative
and retain partially the favorable convergence properties.
This was pursued successfully in
\cite{Dogan-Morin-Nochetto-08,Hintermueller-Ring-03,Hintermueller-Ring-04}
and used to achieve a significant reduction in the number
iterations needed for convergence to the optimal shape.


\bibliographystyle{abbrv}
\bibliography{shape_calculus_Dogan}

\end{document}